\documentclass{amsart}
\usepackage{amssymb}
\usepackage[mathcal]{euscript}
\usepackage[cmtip,all]{xy}
\usepackage{color}

\newcommand{\bydef}{:=}
\newcommand{\defby}{=:}

\newcommand{\vphi}{\varphi}
\newcommand{\veps}{\varepsilon}
\newcommand{\ul}[1]{\underline{#1}}

\newcommand{\wt}[1]{\widetilde{#1}}
\newcommand{\wb}[1]{\overline{#1}}

\newcommand{\id}{\mathrm{id}}

\newcommand{\lspan}[1]{\mathrm{span}\left\{#1\right\}}

\newcommand{\diag}{\mathrm{diag}}

\DeclareMathOperator*{\ot}{\otimes}

\newcommand{\bi}{\mathbf{i}}

\newcommand{\cA}{\mathcal{A}}

\newcommand{\cC}{\mathcal{C}}

\newcommand{\cG}{\mathcal{G}}

\newcommand{\cK}{\mathcal{K}}
\newcommand{\cL}{\mathcal{L}}

\newcommand{\cQ}{\mathcal{Q}}
\newcommand{\cR}{\mathcal{R}}
\newcommand{\cS}{\mathcal{S}}

\newcommand{\cU}{\mathcal{U}}

\newcommand{\frg}{{\mathfrak g}}

\newcommand{\frf}{{\mathfrak f}}

\DeclareMathOperator{\CD}{\mathfrak{CD}}

\newcommand{\ZZ}{\mathbb{Z}}

\newcommand{\RR}{\mathbb{R}}
\newcommand{\CC}{\mathbb{C}}
\newcommand{\HH}{\mathbb{H}}
\newcommand{\OO}{\mathbb{O}}
\newcommand{\FF}{\mathbb{F}}
\newcommand{\KK}{\mathbb{K}}
\newcommand{\chr}[1]{\mathrm{char}\,#1}

\DeclareMathOperator{\Hom}{\mathrm{Hom}}
\DeclareMathOperator{\End}{\mathrm{End}}
\DeclareMathOperator{\Alg}{\mathrm{Alg}}
\DeclareMathOperator{\im}{\mathrm{im}\,}
\DeclareMathOperator{\Aut}{\mathrm{Aut}}
\DeclareMathOperator{\AAut}{\mathbf{Aut}}
\DeclareMathOperator{\Der}{\mathrm{Der}}
\DeclareMathOperator{\supp}{\mathrm{Supp}\,}

\newcommand{\ad}{\mathrm{ad}}
\newcommand{\Ad}{\mathrm{Ad}}

\newcommand{\frso}{{\mathfrak{so}}}
\newcommand{\frpsl}{{\mathfrak{psl}}}

\newcommand{\fru}{{\mathfrak{u}}}
\newcommand{\tri}{\mathfrak{tri}}

\newcommand{\GL}{\mathrm{GL}}

\newcommand{\Ort}{\mathrm{O}}

\newcommand{\PGO}{\mathrm{PGO}}
\newcommand{\Spin}{\mathrm{Spin}}
\newcommand{\TRI}{\mathrm{Tri}}

\newcommand{\Gs}{\mathbf{G}}

\newcommand{\GLs}{\mathbf{GL}}

\newcommand{\PGOs}{\mathbf{PGO}}
\newcommand{\Spins}{\mathbf{Spin}}
\newcommand{\TRIs}{\mathbf{Tri}}

\newcommand{\Qs}{\mathbf{Q}}

\newcommand{\LL}{\mathbb{L}}

\newcommand{\Cl}{\mathfrak{Cl}} 
\newcommand{\TC}{\mathfrak{TC}} 

\newcommand{\Gal}{\mathrm{Gal}} 

\newtheorem{theorem}{Theorem}
\newtheorem{proposition}[theorem]{Proposition}
\newtheorem{lemma}[theorem]{Lemma}
\newtheorem{corollary}[theorem]{Corollary}

\theoremstyle{definition}
\newtheorem{df}[theorem]{Definition}

\theoremstyle{remark}
\newtheorem{remark}[theorem]{Remark}

\begin{document}

\title{Gradings on the simple real Lie algebras of types $G_2$ and $D_4$}

\author[A. Elduque]{Alberto Elduque${}^\star$}
\address{Departamento de Matem\'{a}ticas
 e Instituto Universitario de Matem\'aticas y Aplicaciones,
 Universidad de Zaragoza, 50009 Zaragoza, Spain}
\email{elduque@unizar.es}
\thanks{${}^\star$supported by the Spanish Ministerio de Econom\'{\i}a y Competitividad---Fondo Europeo de Desarrollo Regional (FEDER)  MTM2013-45588-C3-2-P}

\author[M. Kochetov]{Mikhail Kochetov${}^{\star\star}$}
\address{Department of Mathematics and Statistics,
 Memorial University of Newfoundland,
 St. John's, NL, A1C5S7, Canada}
\email{mikhail@mun.ca}
\thanks{${}^{\star\star}$supported by Discovery Grant 341792-2013 of the Natural Sciences and Engineering Research Council (NSERC) of Canada and a grant for visiting scientists by Instituto Universitario de Matem\'aticas y Aplicaciones, University of Zaragoza}

\subjclass[2010]{Primary 17B70; Secondary 17B25, 17C40, 17A75}

\keywords{Graded algebra, real simple Lie algebra, composition algebra, 
cyclic composition, twisted composition, triality, trialitarian algebra}

\date{}

\begin{abstract}
We classify group gradings on the simple Lie algebras of types $G_2$ and $D_4$ over the field of real numbers 
(or any real closed field): fine gradings up to equivalence and $G$-gradings, with a fixed group $G$, up to isomorphism. 
\end{abstract}

\maketitle

\section{Introduction}

Let $\cU$ be an algebra over a field $\FF$ and let $G$ be a group. A \emph{$G$-grading} on $\cU$ is a 
vector space decomposition $\Gamma:\;\cU=\bigoplus_{g\in G}\cU_g$ such that $\cU_g\cU_h\subset\cU_{gh}$ for all $g,h\in G$. 
The nonzero elements $x\in\cU_g$ are said to be \emph{homogeneous of degree} $g$, which can be written as $\deg x=g$, 
and the \emph{support} is the set $\supp\Gamma\bydef\{g\in G\;|\;\cU_g\neq 0\}$. The algebra $\cU$ may have some additional 
structure, e.g. an involution $\sigma$, in which case $\Gamma$ will be required to respect this structure: 
$\sigma(\cU_g)=\cU_g$ for all $g\in G$. 

Group gradings have been extensively studied for many types of algebras --- associative, Lie, Jordan, composition, etc. 
(see e.g. our recent monograph \cite{EKmon} and the references therein). In the case of gradings on simple Lie algebras, 
the support generates an abelian subgroup of $G$, so it is no loss of generality to assume $G$ abelian. 
We will do so for all gradings considered in this paper and will sometimes write $G$ additively. 
For gradings on finitely generated (in particular, finite-dimensional) algebras, we may always replace $G$ with a finitely 
generated group. All algebras in this paper will be assumed finite-dimensional over the ground field $\FF$ unless indicated otherwise.
Unmarked symbol $\otimes$ will refer to the tensor product over $\FF$. 
 
Of particular interest are \emph{fine} gradings, because any $G$-grading can be obtained from at least one fine grading 
by means of a suitable homomorphism from the universal group of the fine grading to $G$ (see details in the next section). 
The classification of fine gradings (up to equivalence) on all finite-dimensional simple Lie algebras 
over an algebraically closed field of characteristic $0$ has recently been completed by the efforts of many authors: 
see \cite[Chapters 3--6]{EKmon}, \cite{YuExc} and \cite{E14}. The classification of all $G$-gradings (up to isomorphism) 
is also known for these algebras, except for types $E_6$, $E_7$ and $E_8$, over an algebraically closed field of 
characteristic different from $2$: see \cite[Chapters 3--6]{EKmon} and \cite{EK15}.

On the other hand, for many applications, algebras over the field of real numbers are especially important. 
Fine gradings on real forms of the classical simple complex Lie algebras were described in \cite{HPP3}, and all $G$-gradings 
have recently been classified in \cite{BKR_Lie}, with the exception of type $D_4$. Fine gradings of the real forms of the simple complex Lie algebras of types $G_2$ and $F_4$ have been classified in \cite{g2f4}.

Simple Lie algebras of type $D_4$ display exceptional behavior due to the phenomenon of triality. 
In particular, classifying them up to isomorphism over an arbitrary field $\FF$ is much more difficult than for types $D_n$
with $n>4$ (see e.g. \cite[\S 45]{KMRT} for an overview, under the assumption $\chr\FF\ne 2$). 
But over the field $\RR$ (or, more generally, any real closed field), 
since it has no field extensions of degree $3$, triality actually results in \emph{fewer} simple Lie algebras of type $D_4$: 
the algebra $\fru^*(8)$ of skew-symmetric elements in $M_4(\HH)$, 
with respect to the involution given by a skew-hermitian form on the $4$-dimensional space over the division algebra 
of quaternions $\HH$, turns out to be isomorphic to the special orthogonal Lie algebra $\frso_{6,2}(\RR)$. 

Another manifestation of triality is the existence of what we call \emph{Type III} gradings (see the definition below) 
on simple Lie algebras of type $D_4$. Over an algebraically closed field $\FF$ of characteristic different from $2$,
gradings on simple Lie algebras of series $A$, $B$, $C$ and $D$ were classified up to isomorphism in \cite{BK10}, 
but type $D_4$ was excluded because in this case, due to triality, not all gradings are ``matrix'', i.e., come 
from gradings on the associative algebra $M_8(\FF)$ with orthogonal involution by restricting 
to the Lie subalgebra of skew-symmetric elements $\frso_8(\FF)$. 
Those that do come in this way are the Type I and Type II gradings. 
All fine gradings for $D_4$ were classified up to equivalence in \cite{E10} over an algebraically closed field 
of characteristic $0$ (see also \cite{DMV}).

Type III gradings cannot be easily seen in the matrix models of simple Lie algebras of type $D_4$. 
In \cite{EK15}, we constructed and classified them over an algebraically closed field $\FF$ with $\chr\FF\ne 2,3$, 
using the model in terms of so-called \emph{trialitarian algebras} \cite[Chapter X]{KMRT}. 
(Type III gradings do not exist if $\chr\FF=3$.) 
Over an algebraically closed field $\FF$ with $\chr\FF\ne 2$, there is, up to isomorphism, only one trialitarian algebra: 
$E=M_8(\FF)\times M_8(\FF)\times M_8(\FF)$ endowed with orthogonal involution $\sigma$ and 
some additional structure $\alpha$ (see the precise definition in the next section). 
The center of $E$ is the cubic separable algebra $\LL=\FF\times\FF\times\FF$. 
The simple Lie algebra $\cL$ of type $D_4$ is canonically embedded in $E$, and it turns out that the restriction map 
$\AAut_\FF(E,\LL,\sigma,\alpha)\to\AAut_\FF(\cL)$ is an isomorphism of affine group schemes  
\cite[Proposition 45.12]{KMRT}. 

Recall that \emph{affine group schemes} over a field $\FF$ are representable functors from the category $\Alg_\FF$ of 
unital associative commutative $\FF$-algebras (not necessarily of finite dimension) to the category of groups --- we refer the reader 
to \cite{Wh}, \cite[Chapter VI]{KMRT} or \cite[Appendix A]{EKmon} for the background. 
We will follow the common convention of denoting the affine group schemes corresponding to classical groups 
by the same letters, but using bold font to distinguish the scheme from the group 
(which is identified with the group of $\FF$-points of the scheme): for example, $\GLs(V)$ and $\GL(V)=\GLs(V)(\FF)$. 
The automorphism group scheme $\AAut_\FF(\cU)$ of a finite-dimensional algebra $\cU$ is defined by 
$\AAut_\FF(\cU)(\cR)=\Aut_\cR(\cU\ot\cR)$ for every $\cR$ in $\Alg_\FF$. 

Gradings by abelian groups often arise as eigenspace decompositions with respect to a family of commuting diagonalizable automorphisms. 
If $\FF$ is algebraically closed and $\chr\FF=0$ then all abelian group gradings on a finite-dimensional algebra 
can be obtained in this way. Over an arbitrary field, a $G$-grading $\Gamma$ on $\cU$ is equivalent to a homomorphism 
of affine group schemes $\eta_\Gamma:G^D\to\AAut_\FF(\cU)$, where $G^D$ is the diagonalizable group scheme 
represented by the Hopf algebra $\FF G$, as follows: for any $\cR$ in $\Alg_\FF$, the corresponding homomorphism of groups 
$(\eta_\Gamma)_\cR:\Alg_\FF(\FF G,\cR)\to\Aut_\cR(\cU\ot\cR)$ is defined by
\begin{equation}\label{eq:def_eta}
(\eta_\Gamma)_\cR(f)(x\ot r)=x\ot f(g)r\;\text{for all}\;x\in\cU_g, g\in G, r\in\cR, f\in\Alg_\FF(\FF G,\cR).
\end{equation}

Now let $\cL$ be a simple Lie algebra of type $D_4$ over a field $\FF$ with $\chr{\FF}\ne 2$. 
Over an algebraic closure $\FF_\mathrm{alg}$, we have the isomorphism 
$\AAut_{\FF_\mathrm{alg}}(\wt{E},\wt{\LL},\sigma,\alpha)\to\AAut_{\FF_\mathrm{alg}}(\wt{\cL})$ 
mentioned above, where $\wt{\cL}=\cL\otimes\FF_\mathrm{alg}$, $\wt{E}$ is the trialitarian algebra over 
$\FF_\mathrm{alg}$, and $\wt{\LL}=\FF_\mathrm{alg}\times\FF_\mathrm{alg}\times\FF_\mathrm{alg}$ is the center of $\wt{E}$. 
It follows, on the one hand, that every grading on the Lie algebra $\wt{\cL}$ is the restriction of 
a unique grading on the trialitarian algebra $\wt{E}$ (which was the starting point of our analysis in \cite{EK15})
and, on the other hand, that $\cL$ is canonically embedded in a unique trialitarian algebra $E$ over $\FF$ 
\cite[Corollary 45.13]{KMRT}. Moreover, we have the isomorphism $\AAut_\FF(E,\LL,\sigma,\alpha)\to\AAut_\FF(\cL)$ and 
hence a bijective correspondence between gradings on the Lie algebra $\cL$ and on the trialitarian algebra $E$. 

Trialitarian algebras over an arbitrary field can be quite complicated. However, over $\RR$ (or any real closed field), 
we have only two possibilities for the cubic separable algebra $\LL$, namely, $\RR\times\RR\times\RR$ and $\RR\times\CC$. 
The trialitarian algebras $E$ over a field $\FF$ whose centers are of the form $\FF\times\FF\times\FF$ or $\FF\times\KK$,
with $\KK$ a quadratic field extension of $\FF$, as well as the corresponding Lie algebras $\cL(E)$, 
are said to be of type ${}^1D_4$ and ${}^2D_4$, respectively. 
Lie algebras of these types were classified by Jacobson \cite{J64}. Here $E$ has the form $A\times\Cl(A)$ where $A$ is a 
central simple associative $\FF$-algebra of degree $8$, endowed with an orthogonal involution, and $\Cl(A)$ is the
corresponding Clifford algebra \cite[Proposition 43.15]{KMRT}. 
The center $Z(E)$ is $\LL=\FF\times Z(\Cl(A))$ and the Lie algebra $\cL(E)$ is isomorphic, 
by means of the projection onto the first factor, to the algebra of skew-symmetric elements in $A$.
Due to \cite[Proposition 43.6]{KMRT}, it turns out that in the case $\FF=\RR$, $E$ is isomorphic to 
$M_8(\RR)\times M_8(\RR)\times M_8(\RR)$, $M_8(\RR)\times M_4(\HH)\times M_4(\HH)$ or $M_8(\RR)\times M_8(\CC)$, and hence we may assume that $A=M_8(\RR)$ with an orthogonal involution of inertia $(p,q)$, $p+q=8$, 
$p\ge q$. Then $\Cl(A)$ is isomorphic to the even part of the Clifford algebra of the corresponding $8$-dimensional 
quadratic space over $\RR$, so we have the following possibilities:

\begin{description}
\item[$\frso_{8,0}(\RR)$] $E=M_8(\RR)\times M_8(\RR)\times M_8(\RR)\simeq\End_\LL(V)$, where  
$\LL=\RR\times\RR\times\RR$, $V=\OO\otimes\LL$, $\OO$ is the division algebra of octonions, 
and the trialitarian structure on $E$ is induced by the \emph{cyclic composition} structure on $V$;

\item[$\frso_{4,4}(\RR)$] the same as above, but with the split algebra of octonions $\OO_s$ 
instead of the division algebra;

\item[$\frso_{7,1}(\RR)$] $E=M_8(\RR)\times M_8(\CC)\simeq\End_\LL(V)$, where $\LL=\RR\times\CC$ and $V$ is a 
\emph{twisted composition} obtained by Galois descent from the complex cyclic composition $\OO\otimes\LL\otimes\CC$ 
as the fixed points of the operator $\bar{\phantom{x}}\otimes\id\otimes\iota$, 
where $\bar{x}=n(x,1)1-x$ is the standard involution of $\OO$ and $\iota$ is the complex conjugation; 

\item[$\frso_{5,3}(\RR)$] the same as above, but with $\OO_s$ instead of $\OO$;

\item[$\frso_{6,2}(\RR)$] $E=M_8(\RR)\times M_4(\HH)\times M_4(\HH)$; this trialitarian algebra is not of the form 
$\End_\LL(V)$ for any composition $V$.
\end{description}

In all cases, we have a short exact sequence of affine group schemes: 
\begin{equation}\label{eq:exact_D4}
\xymatrix{
\mathbf{1}\ar[r] & \PGOs^+_{p,q}\ar[r] & \AAut_\RR(E,\LL,\sigma,\alpha)\ar[r]^-{\pi} & \AAut_\RR(\LL)\ar[r] & \mathbf{1}
}
\end{equation}
where $\pi$ is the restriction map. Indeed, for the groups of $\CC$-points, this is the well-known
exact sequence $1\to\PGO^+_8(\CC)\to\Aut(\frso_8(\CC))\to S_3\to 1$. If $\LL=\RR\times\RR\times\RR$ then $\AAut_\RR(\LL)$ is 
the constant group scheme $\mathbf{S}_3$ corresponding to the group of permutations $S_3$. 
But if $\LL=\RR\times\CC$ then $\AAut_\RR(\LL)$ is the semidirect product $\boldsymbol{\mu}_3\rtimes\mathbf{C}_2$ 
of the group scheme of cubic roots of unity and the constant group scheme corresponding to the cyclic group $C_2$.
Indeed, the group schemes $\AAut_\RR(\RR\times\CC)$ and $\AAut_\RR(\RR\times\RR\times\RR)$ become isomorphic after the extension of scalars 
from $\RR$ to $\CC$, so $\AAut_\RR(\RR\times\CC)$ is a twisted form of $\mathbf{S}_3$. 
Recall that the twisted forms of the constant group scheme $\mathbf{M}$ over a field $\FF$ corresponding to a finite group $M$ 
are given by (continuous) actions of the absolute Galois group of $\FF$ on $M$ (see e.g. \cite[Proposition 20.16]{KMRT}).  
In the case $M=S_3$ and $\FF=\RR$, there is, up to isomorphism, only one nontrivial action, and it gives rise to the indicated semidirect product. 
Alternatively, we can observe that $\AAut_\RR(\RR\times\CC)$ contains $\AAut_\RR(\CC)\simeq \mathbf{C}_2\simeq\boldsymbol{\mu}_2$ 
and a copy of $\boldsymbol{\mu}_3$ associated to the following grading on $\LL=\RR\times\CC$ by the cyclic group 
of order $3$: $\LL=\RR\oplus\RR\xi\oplus\RR\xi^2$ where $\xi=(1,\omega)\in\LL$ and $\omega$ is a fixed primitive cubic root of unity in $\CC$.

Now suppose that we have a $G$-grading $\Gamma$ on the Lie algebra $\cL$. This corresponds to a homomorphism 
$\eta=\eta_\Gamma\colon G^D\to\AAut_\RR(\cL)\simeq\AAut_\RR(E,\LL,\sigma,\alpha)$. 
The image $\pi\eta(G^D)$ is a diagonalizable subgroupscheme of $\AAut_\RR(\LL)$. 
This subgroupscheme must be proper, because $\AAut_\RR(\LL)$ is not diagonalizable (not even abelian). 
Therefore, we have three possibilities: the image has order $1$, $2$ or $3$, and 
the grading $\Gamma$ will be said to have {\em Type I, II or III} accordingly. 
The subgroupscheme $\eta^{-1}(\PGOs^+_{p,q})$ of $G^D$ corresponds to a subgroup $H$ in $G$ 
of order $1$, $2$ or $3$, respectively. We will refer to $H$ as the {\em distinguished subgroup} of $\Gamma$. 
It is the smallest subgroup of $G$ such that the induced $G/H$-grading is of Type I.

The subgroupschemes of $\mathbf{S}_3$ are the constant group schemes corresponding to the subgroups of $S_3$. Since $\mathbf{A}_3\simeq\mathbf{C}_3$ is not 
diagonalizable (over $\RR$), \emph{there are no Type III gradings} in the cases where $\LL=\RR\times\RR\times\RR$. 
On the other hand, in the cases where $\LL=\RR\times\CC$, the (unique) subgroupscheme of order $3$ in $\AAut_\RR(\LL)$ is diagonalizable, 
and we will show that Type III gradings indeed exist by giving explicit models in terms of the Cayley algebras $\OO$ and $\OO_s$. 

Gradings of Types I and II can be studied using matrix models, and all $G$-gradings on (finite-dimensional) central simple $\RR$-algebras with involution 
were classified in \cite{BKR,BKR_Lie} up to isomorphism. 
Hence, our main concern in this paper are Type III gradings on $\frso_{7,1}(\RR)$ and $\frso_{5,3}(\RR)$. 
In order to deal with these, $G$-gradings on $\OO$ and $\OO_s$ are relevant, which will lead us to the classification of gradings on 
the simple Lie algebras of type $G_2$ over $\RR$.

The paper is structured as follows. Section \ref{s:preliminaries} is a long review of the background on group gradings, Galois descent, composition algebras, and trialitarian algebras, which will be used throughout the paper. 
Section \ref{s:octonions_G2} classifies $G$-gradings up to isomorphism, and fine gradings up to equivalence, on Cayley algebras over arbitrary fields. Due to the close relationship between Cayley algebras and algebraic groups of type $G_2$, these results also give the corresponding classifications for simple Lie algebras of type $G_2$ ($\chr \FF\neq 2,3$ here). Section \ref{s:Types_I_II} briefly treats the gradings of Types I and II on simple real Lie algebras of type $D_4$. These are analogous to the gradings on simple real Lie algebras of type $D_n$ for $n\geq 5$ considered in \cite{BKR_Lie}, although triality has a certain effect on the isomorphism problem. Section \ref{s:lifting} deals with the problem of ``lifting'' a Type III grading on a trialitarian algebra to such a grading on a twisted composition. It turns out that, over a real closed field, any Type III grading on the trialitarian algebra $\End_\LL(V)$ attached to a twisted composition $V$ lifts to exactly two gradings on $V$ (Theorem \ref{th:real_tri_to_comp}). This allows us to reduce the problem of classifying the Type III gradings for simple real Lie algebras of type $D_4$ to the corresponding problem for twisted compositions, and this latter is solved in Section \ref{s:Type_III}. Finally, Section \ref{s:Type_III_fine} is devoted to giving concrete descriptions of the Type III gradings on simple real Lie algebras of type $D_4$ by trasferring explicitly our results on twisted compositions to the associated Lie algebras.

\section{Preliminaries}\label{s:preliminaries}

\subsection{Group gradings on algebras}

In the Introduction, we recalled the definition of a grading on an algebra $\cU$ by a group $G$. There is a more general concept of grading: 
a decomposition $\Gamma:\;\cU=\bigoplus_{s\in S}\cU_s$ into nonzero subspaces indexed by a set $S$ and having the property that, 
for any $s_1,s_2\in S$ with $\cU_{s_1}\cU_{s_2}\ne 0$, there exists (unique) $s_3\in S$ such that $\cU_{s_1}\cU_{s_2}\subset\cU_{s_3}$. 
For such a decomposition $\Gamma$, there may or may not exist a group $G$ containing $S$ that makes $\Gamma$ a $G$-grading. 
If such a group (respectively, abelian group) exists, then $\Gamma$ is said to be a {\em group grading} (respectively, an {\em abelian group grading}). 
However, $G$ is usually not unique even if we require that it should be generated by $S$. The {\em universal grading group} 
(or {\em universal abelian grading group}) is generated by $S$ and has the defining relations $s_1s_2=s_3$ for all $s_1,s_2,s_3\in S$ 
such that $0\ne\cU_{s_1}\cU_{s_2}\subset\cU_{s_3}$ (see e.g. \cite[Chapter~1]{EKmon} for details). 
Here we will deal exclusively with abelian group gradings. 

Let $\Gamma:\, \cU=\bigoplus_{g\in G} \cU_g$ and $\Gamma':\,\cU'=\bigoplus_{h\in H} \cU'_h$ be two group gradings, with supports $S$ and $T$, respectively.
We say that $\Gamma$ and $\Gamma'$ are {\em equivalent} if there exists an isomorphism of algebras $\vphi\colon\cU\to\cU'$ 
and a bijection $\alpha\colon S\to T$ such that $\varphi(\cU_s)=\cU'_{\alpha(s)}$ for all $s\in S$. 
If $G$ and $H$ are universal (abelian) grading groups then $\alpha$ extends to an isomorphism $G\to H$. 
In the case $G=H$, the $G$-gradings $\Gamma$ and $\Gamma'$ are {\em isomorphic} if $\cU$ and $\cU'$ are isomorphic as $G$-graded algebras, 
i.e., if there exists an isomorphism of algebras $\vphi\colon\cU\to\cU'$ such that $\varphi(\cU_g)=\cU'_g$ for all $g\in G$. 

If $\Gamma:\,\cU=\bigoplus_{g\in G} \cU_g$ and $\Gamma':\,\cU=\bigoplus_{h\in H} \cU'_h$ are two gradings on the same algebra, 
with supports $S$ and $T$, respectively, then we will say that $\Gamma'$ is a {\em refinement} of $\Gamma$ (or $\Gamma$ is a {\em coarsening} of $\Gamma'$) 
if for any $t\in T$ there exists (unique) $s\in S$ such that $\cU'_t\subset\cU_s$. 
If, moreover, $\cU'_t\ne\cU_s$ for at least one $t\in T$, then the refinement is said to be {\em proper}. 
A grading $\Gamma$ is said to be {\em fine} if it does not admit any proper refinement.

Given a $G$-grading $\Gamma:\,\cU=\bigoplus_{g\in G} \cU_g$, any group homomorphism $\alpha\colon G\to H$ induces an $H$-grading ${}^\alpha\Gamma$ on $\cU$ 
whose homogeneous component of degree $h$ is the sum of all $\cU_g$ with $\alpha(g)=h$. 
Clearly, ${}^\alpha\Gamma$ is a coarsening of $\Gamma$ (not necessarily proper). 
If $G$ is the universal group of $\Gamma$ then every coarsening of $\Gamma$ is obtained in this way. 
If $\Gamma$ and $\Gamma'$ are two gradings, with universal groups $G$ and $H$, then $\Gamma'$ is equivalent to $\Gamma$ if and only if 
$\Gamma'$ is isomorphic to ${}^\alpha\Gamma$ for some group isomorphism $\alpha\colon G\to H$.

\subsection{Galois descent}

Let $\KK$ be a finite Galois field extension of $\FF$ and let $\cG$ be the Galois group.
For any $\FF$-vector space $V$ (not necessarily of finite dimension), there is an action of $\cG$ on $V\otimes\KK$ defined by
$\sigma\cdot(v\otimes\lambda)=v\otimes\sigma(\lambda)$ for all $v\in V$, $\lambda\in\KK$ and $\sigma\in\cG$.
Conversely, if we have an $\FF$-linear action of $\cG$ on a $\KK$-vector space $W$ satisfying 
$\sigma\cdot(w\lambda)=(\sigma\cdot w)\sigma(\lambda)$ for all $w\in W$, $\lambda\in\KK$ and $\sigma\in\cG$, then the set of fixed points $V := W^\cG$ 
is an $\FF$-form of $W$, i.e., $V$ is an $\FF$-subspace and $V\otimes\KK\simeq W$ by means of the map $v\otimes\lambda\mapsto v\lambda$. 
For infinite Galois field extensions, $\cG$ is a (pro-finite) topological group, and this ``descent lemma'' holds if we consider 
\emph{continuous} $\cG$-actions (see e.g. \cite[Lemma 18.1]{KMRT}).   

The following generalization of finite Galois field extensions is often useful (see e.g. \cite[\S 18.B]{KMRT}). 
An \emph{\'{e}tale algebra} over $\FF$ is a separable commutative associative $\FF$-algebra or, in other words, 
a finite direct product of finite separable field extensions of $\FF$. Let $\LL$ be such an algebra 
and let $G$ be a group. Suppose $G$ acts on $\LL$ by $\FF$-linear automorphisms. 
We say that $\LL$ is \emph{$G$-Galois} if $|G|=\dim_\FF\LL$ (known as the degree of $\LL$) and $\LL^G=\FF$.
The ``descent lemma'' holds for free modules over Galois algebras if we use $G$ (rather than $\Aut_\FF(\LL)$) as
a replacement of the Galois group.
(The action of $G$ on $\LL$ gives an embedding $G\to\Aut_\FF(\LL)$, which is an isomorphism if $\LL$ is a field.)

Let $\KK$ be a Galois field extension of $\FF$ (not necessarily finite) and let $\cG=\Gal(\KK/\FF)$. 
The mapping $V\mapsto V\otimes\KK$ is an equivalence from the category of $\FF$-vector spaces to the category of $\KK$-vector spaces 
equipped with a  continuous $\cG$-action as above: a $\KK$-linear map $f\colon V_1\otimes\KK\to V_2\otimes\KK$ ``descends'' 
(i.e., has the form $f_0\otimes\id$ for some $\FF$-linear map $f_0\colon V_1\to V_2$) if and only if $f$ is $\cG$-equivariant. 

Let $\mathbf{S}$ be a functor defined on $\Alg_\FF$. The composition of the forgetful functor $\Alg_\KK\to\Alg_\FF$ with $\mathbf{S}$
is denoted by $\mathbf{S}_\KK$ (so-called ``restriction'' of $\mathbf{S}$). In particular, let $\mathbf{S}$ be an affine scheme over $\FF$,
i.e., a representable functor from $\Alg_\FF$ to the category of sets, and let $\FF[\mathbf{S}]$ denote its representing object. 
Then $\mathbf{S}_\KK$ is an affine scheme over $\KK$, with representing object $\KK[\mathbf{S}] \bydef \FF[\mathbf{S}]\otimes\KK$.
Recall that affine schemes over $\FF$ form a category, with the set of morphisms $\mathrm{Aff}_\FF(\mathbf{S}_1,\mathbf{S}_2)$ consisting of all 
natural transformations $\theta=(\theta_\cR)_{\cR\in\Alg_\FF}$ from $\mathbf{S}_1$ to $\mathbf{S}_2$, and by Yoneda's Lemma this category is dual
to $\Alg_\FF$, i.e., $\mathrm{Aff}_\FF(\mathbf{S}_1,\mathbf{S}_2)\simeq\Alg_\FF(\FF[\mathbf{S}_2],\FF[\mathbf{S}_1])$. 
The element of $\Alg_\FF(\FF[\mathbf{S}_2],\FF[\mathbf{S}_1])$ (``comorphism'') corresponding to a morphism 
$\theta\in\mathrm{Aff}_\FF(\mathbf{S}_1,\mathbf{S}_2)$ is denoted by $\theta^*$; 
we have $\theta_\cR(f)=f\circ\theta^*$ for all $f\in\mathbf{S}_1(\cR)=\Alg_\FF(\FF[\mathbf{S}_1],\cR)$.  
The mapping $\mathbf{S}\mapsto\mathbf{S}_\KK$ is functorial, and a morphism $\theta\colon(\mathbf{S}_1)_\KK\to(\mathbf{S}_2)_\KK$ ``descends'' 
(i.e., comes from a morphism $\mathbf{S}_1\to\mathbf{S}_2$) if and only if the comorphism $\theta^*\colon\KK[\mathbf{S}_2]\to\KK[\mathbf{S}_1]$ 
is $\cG$-equivariant, or, equivalently, $\theta^*$ is a fixed point of the $\cG$-action on $\Alg_\KK(\KK[\mathbf{S}_2],\KK[\mathbf{S}_1])$
defined by $(\sigma\cdot f)(a)=\sigma\cdot(f(\sigma^{-1}\cdot a))$ for all $a\in\KK[\mathbf{S}_2]$. 
This latter action is continuous if $\cG$ is finite or $\FF[\mathbf{S}_2]$ is finitely generated.
It is transported to $\mathrm{Aff}_\KK((\mathbf{S}_1)_\KK,(\mathbf{S}_2)_\KK)$ through the identification with 
$\Alg_\KK(\KK[\mathbf{S}_2],\KK[\mathbf{S}_1])$, and it is compatible with composition of morphisms.

The above remarks apply to affine group schemes over $\FF$, the only difference being that, for an affine group scheme $\Gs$, 
the representing object $\FF[\Gs]$ is a (commutative) Hopf algebra. For example, if $\Gs=\AAut_\FF(\cU)$ then $\Gs_\KK=\AAut_\KK(\cU\otimes\KK)$. 
Note that the set $\mathrm{Aff}_\FF(\Gs_1,\Gs_2)\simeq\Gs_2(\FF\Gs_1)$ is a group, which contains the set of homomorphisms of affine group schemes 
$\mathrm{Hom}_\FF(\Gs_1,\Gs_2)$. This latter is a subgroup if $\Gs_2$ is abelian. The action of $\cG$ on 
$\mathrm{Aff}_\KK((\Gs_1)_\KK,(\Gs_2)_\KK)$ is by group automorphisms, and $\mathrm{Hom}_\KK((\Gs_1)_\KK,(\Gs_2)_\KK)$
is an invariant subset. In particular, $\cG$ acts on $\Gs_\KK(\KK)\simeq\mathrm{Aff}_\KK(\mathbf{1}_\KK,\Gs_\KK)$, which can be identified with 
the group of $\KK$-points of $\Gs$, $\Gs(\KK)=\Alg_\FF(\FF[\Gs],\KK)$, where $\cG$ acts through its action on $\KK$. 
If $\theta\in\mathrm{Aff}_\KK((\Gs_1)_\KK,(\Gs_2)_\KK)$ is a fixed point of $\cG$ then so is $\theta_\KK\colon\Gs_1(\KK)\to\Gs_2(\KK)$, because the mapping $\theta\mapsto\theta_\KK$ is $\cG$-equivariant (and a group homomorphism). 
The converse holds in the following special case: $\chr{\FF}=0$ and $\KK=\FF_\mathrm{alg}$, because $\Gs_1(\FF_\mathrm{alg})$ separates points of $\FF[\Gs_1]$ 
and hence the mapping $\theta\mapsto\theta_{\FF_\mathrm{alg}}$ is injective.

\subsection{Composition algebras}

Recall that a (finite-dimensional) {\em composition algebra} over a field $\FF$ is an algebra $\cA$ with a nonsingular quadratic form $n$ 
such that $n(xy)=n(x)n(y)$ for all $x,y\in\cA$. It is well known that $\dim\cA$ can only be $1$, $2$, $4$ or $8$. 
The unital composition algebras are called {\em Hurwitz algebras} and can be obtained using the Cayley--Dickson doubling process. 
Hurwitz algebras of dimension $4$ are called \emph{quaternion algebras} and those of dimension $8$ are called {\em octonion}, or {\em Cayley algebras}.
(Cayley algebras are alternative, while Hurwitz algebras of dimension $\le 4$ are associative.) 
A Hurwitz algebra $\cA$ is a division algebra if and only if the norm does not represent $0$, i.e., $n(x)\ne 0$ for all $0\ne x\in\cA$.
Otherwise $\cA$ is said to be \emph{split}. It is well known that, up to isomorphism, there is a unique split Hurwitz algebra in each dimension 
different from $1$. (In dimension $1$, the unique Hurwitz algebra is $\FF$.) 
In particular, over an algebraically closed field, there is only one Hurwitz algebra in each dimension.
If $\FF=\RR$ then there are two Hurwitz algebras in dimensions different from $1$: one division and the other split, namely, $\CC$ and $\CC_s \simeq \RR\times\RR$ 
in dimension $2$, $\HH$ and $\HH_s \simeq M_2(\RR)$ in dimension $4$, and $\OO$ and $\OO_s$ in dimension $8$. 
We will also need another kind of composition algebras.

\begin{df}\label{df:symmetric}
A composition algebra $\cS$, with multiplication $\star$ and norm $n$, is said to be \emph{symmetric} if the polar form of the norm, 
$n(x,y)\bydef n(x+y)-n(x)-n(y)$, is associative:
\[
n(x\star y,z)=n(x,y\star z),
\]
for all $x,y,z\in\cS$.
\end{df}

As a consequence, $\cS$ satisfies the following identities \cite{OkOsb}:
\begin{equation}\label{eq:symmetric_comp_id}
(x\star y)\star x=n(x)y=x\star(y\star x).
\end{equation}

Any Hurwitz algebra gives a symmetric composition algebra with respect to the so-called \emph{para-Hurwitz product}: 
$x\star y=x\bullet y\bydef \bar{x}\bar{y}$, where  $\bar{x}\bydef n(x,1)1-x$ is the standard involution.
Note that, with respect to the para-Hurwitz product, $1$ becomes a \emph{para-unit}, i.e., an idempotent $\veps$ satisfying 
$\veps\bullet x=x\bullet\veps=n(x,\veps)\veps-x$.

From now on, we assume $\chr{\FF}\ne 2$.

\begin{df}\label{df:tri}
Let $\cS$ be a symmetric composition algebra of dimension $8$. Its \emph{triality Lie algebra} is defined as
\[
\tri(\cS,\star,n)=\{(d_1,d_2,d_3)\in\frso(\cS,n)^3\;|\; d_1(x\star y)=d_2(x)\star y+x\star d_3(y)\ \forall x,y\in\cS\}.
\]
This is a Lie algebra with componentwise bracket.
\end{df}

It turns out that this definition is symmetric with respect to cyclic permutations of $(d_1,d_2,d_3)$, 
and each projection determines an isomorphism $\tri(\cS,\star,n)\to\frso(\cS,n)$, so $\tri(\cS,\star,n)$ is a Lie algebra of type $D_4$ 
(see e.g. \cite[\S 5.5, \S 6.1]{EKmon} or \cite[\S 45.A]{KMRT}, but note that in the latter the ordering of triples differs from ours). 
This fact is known as the ``local triality principle''. There is also a ``global triality principle'', as follows.

\begin{df}\label{df:Tri}
Let $\cS$ be a symmetric composition algebra of dimension $8$. Its \emph{triality group} is defined as
\[
\TRI(\cS,\star,n)=\{(f_1,f_2,f_3)\in\Ort(\cS,n)^3\;|\; f_1(x\star y)=f_2(x)\star f_3(y)\ \forall x,y\in\cS\}.
\]
This is an algebraic group with componentwise multiplication, and its tangent algebra is $\tri(\cS,\star,n)$.
\end{df}

It turns out that this definition, too, is symmetric with respect to cyclic permutations of $(f_1,f_2,f_3)$, and $\TRI(\cS,\star,n)$ is isomorphic to $\Spin(\cS,n)$. 
In fact, this isomorphism can be defined at the level of the corresponding group schemes (see \cite[\S 35.C]{KMRT} for details). 
The said cyclic permutations determine outer actions of $A_3$ on $\Spin(\cS,n)$ and its Lie algebra $\frso(\cS,n)$. 
If $\cS$ is a para-Cayley algebra then one can define an outer action of $S_3$ using $(f_1,f_2,f_3)\mapsto(\bar{f}_1,\bar{f}_3,\bar{f}_2)$ 
and $(d_1,d_2,d_3)\mapsto(\bar{d}_1,\bar{d}_3,\bar{d}_2)$ as the action of the transposition $(2,3)$, 
where $\bar{f}$ is defined by $\bar{f}(\bar{x})=\overline{f(x)}$. 
The $S_3$-action on $\Spins(\cS,n)$ passes to the quotient $\PGOs^+(\cS,n)$, and 
this allows us to split the exact sequence \eqref{eq:exact_D4} if $(p,q)=(8,0)$ or $(4,4)$.

\subsection{Cyclic compositions}\label{ss:cyclic_compositions}

A convenient way to ``package'' triples of maps as above is the following concept due to Springer \cite{Springer}. 
Let $\LL$ be a Galois algebra over $\FF$ with respect to the cyclic group of order $3$. 
Fix a generator $\rho$ of this group. For any $\ell\in\LL$, we have the norm $N(\ell)=\ell\rho(\ell)\rho^2(\ell)$, 
the trace $T(\ell)=\ell+\rho(\ell)+\rho^2(\ell)$, and the adjoint $\ell^\sharp=\rho(\ell)\rho^2(\ell)$. 
Our main interest is in the case $\LL=\FF\times\FF\times\FF$ where $\rho(\ell_1,\ell_2,\ell_3)=(\ell_2,\ell_3,\ell_1)$.

\begin{df}\label{df:cyclic_comp}
A {\em cyclic composition} over $(\LL,\rho)$ is a free $\LL$-module $V$ with a nonsingular $\LL$-valued quadratic form $Q$ and 
an $\FF$-bilinear multiplication $(x,y)\mapsto x*y$ that is $\rho$-semilinear in $x$ and $\rho^2$-semilinear in $y$ and satisfies the following identities:
\begin{align*}
Q(x*y)&=\rho(Q(x))\rho^2(Q(y)),\\
b_Q(x*y,z)&=\rho(b_Q(y*z,x))=\rho^2(b_Q(z*x,y)),
\end{align*}
where $b_Q(x,y)\bydef Q(x+y)-Q(x)-Q(y)$ is the polar form of $Q$.
An {\em isomorphism} from $(V,\LL,\rho,*,Q)$ to $(V',\LL',\rho',*',Q')$ is a pair of $\FF$-linear isomorphisms 
$\varphi_0\colon(\LL,\rho)\to(\LL',\rho')$ (i.e., $\varphi_0$ is an isomorphism that satisfies $\varphi_0\rho=\rho'\varphi_0$) and 
$\varphi_1\colon V\to V'$ such that $\varphi_1$ is $\varphi_0$-semilinear, $\varphi_1(x*y)=\varphi_1(x)*'\varphi_1(y)$ and 
$\varphi_0(Q(x))=Q'(\varphi_1(x))$ for all $x,y\in V$.
\end{df}

As a consequence, $V$ also satisfies
\begin{equation}\label{eq:cyclic_comp_id}
(x*y)*x=\rho^2(Q(x))y\quad\text{and}\quad x*(y*x)=\rho(Q(x))y.
\end{equation}

If $(\cS,\star,n)$ is a symmetric composition algebra then $\cS\ot\LL$ becomes a cyclic composition with $Q(x\ot\ell)=n(x)\ell^2$ 
(extended to sums in the obvious way using the polar form of $n$) and $(x\ot\ell)*(y\ot m)=(x\star y)\ot\rho(\ell)\rho^2(m)$. 
With $\LL=\FF\times\FF\times\FF$ and $\rho(\ell_1,\ell_2,\ell_3)=(\ell_2,\ell_3,\ell_1)$, this gives $V=\cS\times\cS\times\cS$ with $Q=(n,n,n)$ and
\begin{equation}\label{df:cyclic_prod}
(x_1,x_2,x_3)*(y_1,y_2,y_3)=(x_2\star y_3,x_3\star y_1,x_1\star y_2).
\end{equation}
For $\cS$ of dimension $8$, the multiplication \eqref{df:cyclic_prod} allows us to interpret $\TRI(\cS,\star,n)$ as $\Aut_\LL(V,*,Q)$, 
the group of $\LL$-linear automorphisms (i.e., with $\varphi_0=\id$), and $\TRI(\cS,\star,n)\rtimes A_3$ as $\Aut_\FF(V,\LL,\rho,*,Q)$, 
the group of all automorphisms (see Definition \ref{df:cyclic_comp}). This interpretation can be made at the level of group schemes:
\[
\AAut_\FF(V,\LL,\rho,*,Q)=\TRIs(\cS,\star,n)\rtimes\mathbf{A}_3\simeq\Spins(\cS,n)\rtimes\mathbf{A}_3.
\]
Similarly, $\tri(\cS,\star,n)$ can be interpreted as $\Der_\LL(V,*,Q) \bydef \{d\in\frso(V,Q)\;|\;d(x*y)=d(x)*y+x*d(y)\ \forall x,y\in V\}$.

\subsection{Twisted compositions}\label{ss:twisted_compositions}

In order to deal with the case $\LL=\RR\times\CC$, we will need the following generalization of cyclic compositions. 
Let $\LL$ be a cubic \'{e}tale algebra over $\FF$, with the norm $N=N_{\LL/\FF}\colon\LL\to\FF$, 
the trace $T=T_{\LL/\FF}\colon\LL\to\FF$, and the adjoint $\sharp\colon\LL\to\LL$.
(The latter is the unique quadratic map satisfying $\ell\ell^{\sharp}=N(\ell)$ for all $\ell\in\LL$.) 
For example, if $\LL=\RR\times\CC$ then, for any $\ell=(b,c)$, we have $N(\ell)=bc\bar{c}$, $T(\ell)=b+c+\bar{c}$, and $\ell^{\sharp}=(c\bar{c},b\bar{c})$.

\begin{df}\label{df:twisted_comp}
A {\em twisted composition} over $\LL$ is a free $\LL$-module $V$ with a nonsingular $\LL$-valued quadratic form $Q$ and 
an $\FF$-quadratic map $\beta\colon V\to V$ that satisfy the following conditions:
\begin{align*}
\beta(\ell v)&=\ell^{\sharp}\beta(v),\\
Q(\beta(v))&=Q(v)^{\sharp},\\
N_V(v) &\bydef b_Q(v,\beta(v)) \in \FF,
\end{align*}
for all $v\in V$ and $\ell\in\LL$.
An {\em isomorphism} from $(V,\LL,\beta,Q)$ to $(V',\LL',\beta',Q')$ is a pair of $\FF$-linear isomorphisms 
$\varphi_0\colon \LL\to\LL'$ and 
$\varphi_1\colon V\to V'$ such that $\varphi_1$ is $\varphi_0$-semilinear, $\varphi_1(\beta(v))=\beta'(\varphi_1(v))$ and 
$\varphi_0(Q(v))=Q'(\varphi_1(v))$ for all $v\in V$.
\end{df}

This concept generalizes cyclic compositions in the following sense (see \cite[Proposition 36.12]{KMRT}): 
if $\LL$ is Galois with respect to the cyclic group of order $3$ and $\rho_i$, $i=1,2$, are the nontrivial elements of this group 
then any cyclic composition $(V,\LL,\rho_i,*_i,Q)$ yields a twisted composition $(V,\LL,\beta,Q)$
by setting $\beta(v) := v *_i v$ for all $v\in V$ and, conversely, for any twisted composition $(V,\LL,\beta,Q)$, there is a unique pair of cyclic compositions
$(V,\LL,\rho_i,*_i,Q)$, $i=1,2$, such that $\beta(v) = v *_i v$ for all $v\in V$. Moreover, these cyclic compositions are the opposites of one another:
$x *_1 y = y *_2 x$ for all $x,y\in V$.

If $(V,\LL,\beta,Q)$ is a twisted composition and $\lambda,\mu\in\LL^\times$, one checks that $\tilde{\beta}(v) := \lambda\beta(v)$ and 
$\tilde{Q}(v) := \mu Q(v)$ define a twisted composition if and only if $\mu=\lambda^\sharp$ (see \cite[Lemma 36.1]{KMRT}). 
We will say that an isomorphism $(V,\LL,\tilde{\beta},\tilde{Q})\to(V',\LL',\beta',Q')$ is a {\em similitude} from $(V,\LL,\beta,Q)$ to $(V',\LL',\beta',Q')$
with {\em parameter} $\lambda$ and {\em multiplier} $\lambda^\sharp$. In particular, for any $\ell\in\LL^\times$, 
the mappings $\varphi_0=\id$ and $\varphi_1(v)=\ell v$ define a similitude from $(V,\LL,\beta,Q)$ to itself with parameter 
$\ell^{-1}\ell^\sharp$ and multiplier $\ell^2$.

It is known (see \cite[Proposition 36.3]{KMRT}) that, for $\LL=\FF\times\FF\times\FF$, any twisted composition is similar 
to the one determined by the cyclic composition $\cS\ot\LL$, with multiplication given by Equation \eqref{df:cyclic_prod}, 
where $(\cS,\star,n)$ is a para-Hurwitz algebra. 
(Using extension of scalars, it follows that, for any $\LL$, the $\LL$-rank of a twisted composition can be $1$, $2$, $4$ or $8$.) 

We will need the following generalization of the construction $\cS\ot\LL$, with $\cS$ para-Hurwitz, to an arbitrary cubic \'{e}tale algebra $\LL$.
Let $\Delta$ be the discriminant of $\LL$ and let $\Sigma\simeq \LL\ot\Delta$ be the $S_3$-Galois closure of $\LL$ (see e.g. \cite[\S 18B, \S 18C]{KMRT}).
By definition, $\Delta$ consists of the fixed points of the subgroup $A_3\subset S_3$ in $\Sigma$, so $\Delta$ is a quadratic \'{e}tale algebra over $\FF$ 
and $\Sigma$ is $A_3$-Galois over $\Delta$. Note that $\Delta\simeq\FF\times\FF$ if $\LL$ admits the structure of a Galois algebra over $A_3$, and $\Delta$ is 
a quadratic field extension of $\FF$ otherwise.
Using the $A_3$-action on $\Sigma$, we can define the structure of cyclic composition on $\cS\otimes\Sigma\simeq\cS\ot\LL\otimes\Delta$ for each of the $3$-cycles $\rho$ in $A_3$, 
and they both determine the same twisted composition $(\cS\otimes\Sigma,\Sigma,\beta,Q)$. We want to descend from $\Sigma$ to $\LL$ and from $\Delta$ to $\FF$.
To this end, let $\iota$ be the generator of the Galois group of $\Delta$ over $\FF$. 
It is the restriction of any of the three transpositions in $S_3$ acting on $\Sigma$.
There are actually three ways to identify $\Sigma$ with $\LL\otimes\Delta$ because, for each transposition, the algebra of fixed points in $\Sigma$ is 
canonically isomorphic to $\LL$. Fix one transposition, say, $(2,3)$, and the corresponding identification. Then $(2,3)$ acts on $\LL\otimes\Delta$ as 
$\id\otimes\iota$. It follows that the involutive operator $\tilde{\iota} := \bar{}\otimes\id\otimes\iota$ on $\cS\otimes\LL\otimes\Delta$ determines a 
Galois descent from $\LL\otimes\Delta$ to $\LL$, so
\begin{equation}\label{eq:twisted_comp}
V := \{x\in\cS\otimes\LL\otimes\Delta\;|\;\tilde{\iota}(x)=x\}
\end{equation}
is an $\LL$-form. Since $\tilde{\iota}(x*y)=\tilde{\iota}(y)*\tilde{\iota}(x)$, the map $\beta(x)=x*x$ restricts to $V$. 
Also, the restriction of $Q$ to $V$ takes values in $\LL$. Thus we obtain a twisted composition $(V,\LL,\beta,Q)$, 
which we will denote by $\TC(\cS,\LL)$. A twisted composition similar to some $\TC(\cS,\LL)$ is said to be a
\emph{twisted Hurwitz composition}. 

It is known (see e.g. \cite[Corollary 36.29]{KMRT}) that if $\LL$ is not a field --- 
in other words, $\LL=\FF\times\KK$ where $\KK$ is a quadratic \'{e}tale algebra over $\FF$ --- then any twisted composition over $\LL$ is a twisted Hurwitz composition. Note that in this case $\Delta\simeq\KK$. Specifically for $\LL=\RR\times\CC$, we have $\Delta\simeq\CC$ and 
$\Sigma\simeq\CC\times\CC\times\CC$, with $S_3$ acting by $\rho\cdot(c_1,c_2,c_3)=(c_{\rho^{-1}(1)},c_{\rho^{-1}(2)},c_{\rho^{-1}(3)})$ for a $3$-cycle $\rho$ 
and $(2,3)\cdot(c_1,c_2,c_3)=(\bar{c}_1,\bar{c}_3,\bar{c}_2)$; 
$\LL$ is embedded in $\CC\times\CC\times\CC$ by $(b,c)\mapsto (b,c,\bar{c})$.

For any twisted composition $(V,\LL,\beta,Q)$ of rank $8$, the affine group scheme $\AAut_\LL(V,\beta,Q)$ of $\LL$-linear automorphisms is a twisted form of 
$\Spins_8$ (which arises when $V=\cS\otimes\LL$ where $\LL=\FF\times\FF\times\FF$ and $\cS$ is the para-Hurwitz algebra associated to the split Cayley algebra) 
and is denoted by $\Spins(V,\LL,\beta,Q)$ (see \cite[\S 36.A]{KMRT}). We have a short exact sequence: 
\begin{equation}\label{eq:exact_twisted_comp}
\xymatrix{
\mathbf{1}\ar[r] & \Spins(V,\LL,\beta,Q)\ar[r] & \AAut_\FF(V,\LL,\beta,Q)\ar[r] & \AAut_\FF(\LL)\ar[r] & \mathbf{1}
}
\end{equation}
The tangent algebra of $\Spins(V,\LL,\beta,Q)$ is 
\[
\Der_\LL(V,\beta,Q) \bydef \{d\in\frso(V,Q)\;|\;d(\beta(x,y))=\beta(d(x),y)+\beta(x,d(y))\ \forall x,y\in V\},
\]
which is a simple Lie algebra of type $D_4$ (see \cite[\S 45.B]{KMRT}). Here $\beta(x,y) \bydef \beta(x+y)-\beta(x)-\beta(y)$.

\subsection{Trialitarian algebras}

The definition is given in \cite[\S 43.A]{KMRT} and requires some preparation. 
First, if $A$ is a central simple (associative) algebra over $\FF$, $\chr{\FF}\ne 2$, 
and $\sigma$ is an orthogonal involution on $A$ then one can define the \emph{Clifford algebra} $\Cl(A,\sigma)$ as the quotient of the tensor algebra of 
the space $A$ modulo certain relations (see \cite[\S 8.B]{KMRT}). If $A=\End_\FF(V)$, where $V$ is a (finite-dimensional) vector space over $\FF$, 
and $\sigma$ is induced by a quadratic form $Q$ on $V$ then $\Cl(A,\sigma)\simeq\Cl_0(V,Q)$, the even part of the well-known Clifford algebra of a quadratic space.
The inclusion of $A$ into its tensor algebra yields a canonical $\FF$-linear map $\kappa\colon A\to\Cl(A,\sigma)$, whose image generates $\Cl(A,\sigma)$, 
but which is neither injective nor a homomorphism of algebras. This construction is functorial for isomorphisms of algebras with involution:  given such an isomorphism $\varphi\colon(A,\sigma)\to(A',\sigma')$, there is a unique isomorphism $\Cl(\varphi)\colon\Cl(A,\sigma)\to\Cl(A',\sigma')$ such that $\kappa'\varphi=\Cl(\varphi)\kappa$.
In particular, $\sigma$ determines a unique involution $\ul{\sigma}=\Cl(\sigma)$ on $\Cl(A,\sigma)$.  
Also, any action (respectively, grading) by a group on the algebra with involution $(A,\sigma)$ gives rise to a unique action (respectively, grading) 
on the algebra with involution $(\Cl(A,\sigma),\ul{\sigma})$ such that $\kappa$ is equivariant (respectively, preserves the degree). 
The isomorphism $\Cl_0(V,Q)\to\Cl(A,\sigma)$ mentioned above in the case $A=\End_\FF(V)$ sends, for any $x,y\in V$, 
the element $x\cdot y\in\Cl_0(V,Q)$ to the image of the operator $z\mapsto xb_Q(y,z)$ under $\kappa$. 
We will also need a generalization of the Clifford algebra construction for a separable associative algebra with involution: 
the center is an \'{e}tale algebra $\LL$ over $\FF$, i.e., the direct product of (finite) separable field extensions of $\FF$, 
and we require that the involution be $\LL$-linear and orthogonal on each factor, so we can apply the above considerations 
separately over each factor.

Next, we introduce an analog of Definitions \ref{df:tri} and \ref{df:Tri} for triples of central simple algebras of 
degree $8$ over $\FF$. Let $(\cS,\bullet,n)$ be the para-Hurwitz algebra associated to the split Cayley algebra over $\FF$. Recall that it gives an $S_3$-action on $\PGOs^+(\cS,n)=\PGOs_8^+$. In particular, $A_3$ acts by cyclic permutations
of the following triples $(\psi,\psi^+,\psi^-)$ of elements of 
$\PGO^+(\cS_\mathrm{sep},n_\mathrm{sep})\subset\Aut_{\FF_\mathrm{sep}}(R_\mathrm{sep},\sigma_\mathrm{sep})$,
where $\FF_\mathrm{sep}$ is a separable closure of $\FF$, $R=\End_\FF(\cS)$, $\sigma$ is the involution of $R$
induced by $n$, and the subscript $\mathrm{sep}$ denotes objects obtained by extending scalars 
from $\FF$ to $\FF_\mathrm{sep}$. Identities \eqref{eq:symmetric_comp_id} imply that the mapping
\[
x\mapsto\begin{pmatrix}0 & r_x \\ l_x & 0\end{pmatrix}\in\End_{\FF}\big(\cS\oplus\cS\big),\;x\in\cS,
\]
where $l_x(y)\bydef x\bullet y\defby r_y(x)$, extends to an isomorphism of $\ZZ_2$-graded algebras with involution
\[
\alpha_\cS\colon(\Cl(\cS,n),\tau)\stackrel{\sim}{\to}(\End_{\FF}\big(\cS\oplus\cS\big),\tilde\sigma),
\]
where $\tau$ is the standard involution of the Clifford algebra and $\tilde\sigma$ is induced by the quadratic form
$n\perp n$ on $\cS\oplus\cS$. Restricting to the even part and using the above isomorphism $\Cl_0(\cS,n)\simeq\Cl(R,\sigma)$,
we obtain the isomorphism
\[
\alpha_R\colon(\Cl(R,\sigma),\ul{\sigma})\stackrel{\sim}{\to}(R,\sigma)\times (R,\sigma).
\]
Given $\psi\in\PGO^+(\cS_\mathrm{sep},n_\mathrm{sep})$, the corresponding 
$\psi^+$ and $\psi^-$ are determined by the condition 
$(\psi^+\times\psi^-)\alpha_{R_\mathrm{sep}}=\alpha_{R_\mathrm{sep}}\Cl(\psi)$.
The $S_3$-action on $\PGO^+(\cS_\mathrm{sep},n_\mathrm{sep})$ yields an $S_3$-action on continuous $1$-cocycles 
$\cG\to\PGO^+(\cS_\mathrm{sep},n_\mathrm{sep})$ where $\cG=\mathrm{Gal}(\FF_\mathrm{sep}/\FF)$. 
Each such $1$-cocycle $\gamma$ twists the standard action of $\cG$ on $(R_\mathrm{sep},\sigma_\mathrm{sep})$ 
and hence determines a twisted form $(A,\sigma_A)$ of $(R,\sigma)$. Similarly, $\gamma^+$ and $\gamma^-$ determine 
twisted forms $(B,\sigma_B)$ and $(C,\sigma_C)$, respectively. This situation can be described in more abstract terms:  
a \emph{related triple of algebras} over $\FF$ is an ordered triple $(A,B,C)$, where $A$, $B$ and $C$ are central simple algebras of degree $8$ with orthogonal involutions $\sigma_A$, $\sigma_B$ and $\sigma_C$, respectively, together with an isomorphism of algebras with involution $\alpha_A\colon (\Cl(A,\sigma_A),\ul{\sigma}_A)\stackrel{\sim}{\to}(B,\sigma_B)\times(C,\sigma_C)$. An \emph{isomorphism} $((A,B,C),\alpha_A)\to((A',B',C'),\alpha_{A'})$ is a triple of isomorphisms 
$\varphi_1\colon(A,\sigma_A)\to(A',\sigma_{A'})$, $\varphi_2\colon(B,\sigma_B)\to(B',\sigma_{B'})$ and 
$\varphi_3\colon(C,\sigma_C)\to(C',\sigma_{C'})$ such that $(\varphi_2\times\varphi_3)\alpha_A=\alpha_A\Cl(\varphi_1)$.
If $A$, $B$ and $C$ are the $\FF$-forms of $R_\mathrm{sep}$ as above, then $\alpha_A$ is just the restriction of 
$\alpha_{R_\mathrm{sep}}$. Conversely, any related triple of algebras is isomorphic to one of this form.
It follows that there is an $A_3$-action on related triples of algebras: the algebras $A$, $B$ and $C$ are permuted 
cyclically and the isomorphism $\alpha_A$ determines analogous isomorphisms 
$\alpha_B\colon (\Cl(B,\sigma_B),\ul{\sigma}_B)\stackrel{\sim}{\to}(C,\sigma_C)\times(A,\sigma_A)$ 
and $\alpha_C\colon (\Cl(C,\sigma_C),\ul{\sigma}_C)\stackrel{\sim}{\to}(A,\sigma_A)\times(B,\sigma_B)$.
Actually, we have an $S_3$-action on related triples of algebras: $(2,3)$ sends $(A,B,C)$ to $(A,C,B)$ and 
$\alpha_A$ to its composition with the flip $B\times C\to C\times B$.

\begin{remark}
Our treatment here differs from \cite[\S 42.A]{KMRT}, where $\alpha_A$ is not fixed and the ordering of triples is different. Thus, our related triple $((A,B,C),\alpha_A)$ corresponds to a \emph{trialitarian triple} $(A,C,B)$ in 
\cite{KMRT}.
\end{remark}

Now consider the following objects: $(E,\LL,\rho,\sigma,\alpha)$ where $\LL$ is a cubic \'{e}tale algebra, $\rho$ is a generator of $A_3$, 
$E$ is an associative algebra with center $\LL$ such that, for each primitive idempotent $\varepsilon\in\LL$, the algebra $\varepsilon E$ is central simple 
of degree $8$ over $\varepsilon\LL$, $\sigma$ is an $\LL$-linear involution of $E$ that is orthogonal on each $\varepsilon E$, and finally 
\[
\xymatrix{
\alpha\colon\big(\Cl(E,\sigma),\ul{\sigma}\big)\ar[r]^\sim & (E\otimes\Delta,\sigma\otimes\id)^\rho
}
\]
is an isomorphism of $\LL$-algebras with involution, where $\Delta$ is the discriminant of $\LL$ and the superscript $\rho$ 
denotes the twist of the $(\LL\otimes\Delta)$-module structure through the action of $\rho$ on $\LL\otimes\Delta$. 
 Recall that we identified $\LL\otimes\Delta$ with the $S_3$-Galois closure of $\LL$ so there is an action of $S_3$ on $\LL\otimes\Delta$ such that $\LL\simeq\LL\otimes 1$ is the set of fixed points of $(2,3)$ and $\Delta\simeq 1\otimes\Delta$ is the set of fixed points of $A_3$. 
An {\em isomorphism} $(E,\LL,\rho,\sigma,\alpha)\to(E',\LL',\rho',\sigma',\alpha')$ is defined to be an isomorphism 
$\varphi\colon (E,\sigma)\to (E',\sigma')$ of $\FF$-algebras with involution such that the following diagram commutes:
\[
\xymatrix{
\Cl(E,\sigma)\ar[r]^\alpha\ar[d]_{\Cl(\varphi)} & (E\otimes\Delta)^\rho\ar[d]^{\varphi\ot\delta}\\
\Cl(E',\sigma')\ar[r]^{\alpha'} & (E'\otimes\Delta')^{\rho'}
}
\]
where $\delta\colon\Delta\to\Delta'$ is the isomorphism induced by the restriction $\varphi\colon\LL\to\LL'$.

If $\LL=\FF\times\FF\times\FF$ then $\Delta=\FF\times\FF$ and $\LL\otimes\Delta\simeq\LL\times\LL$
where the action of $(2,3)$ on $\LL\otimes\Delta$ interchanges the two components and hence  
$\rho$ acts on them in opposite ways; we order the components so that $\rho$ acts as itself on the second and as $\rho^2$ on the first.
Consequently, $(E\otimes\Delta)^\rho\simeq E^{\rho^2}\times E^{\rho}$. 
If we write $E=E_1\times E_2\times E_3$ and $\sigma=\sigma_1\times\sigma_2\times\sigma_3$ then, say, for $\rho=(1,3,2)$, 
the isomorphism $\alpha$ can be regarded as a triple of isomorphisms:
$\alpha_i\colon(\Cl(E_i,\sigma_i),\ul{\sigma}_i)\to(E_{i+1},\sigma_{i+1})\times(E_{i+2},\sigma_{i+2})$, with subscripts modulo $3$.
These isomorphisms make $(E_1,E_2,E_3)$, $(E_2,E_3,E_1)$ and $(E_3,E_1,E_2)$ related triples of algebras.
We say that $(E,\FF\times\FF\times\FF,\rho,\sigma,\alpha)$  is a \emph{trialitarian algebra} if these 
triples form an $A_3$-orbit.

Note that this condition is not affected by extension of scalars. For an arbitrary $\LL$, we declare $(E,\LL,\rho,\sigma,\alpha)$ a \emph{trialitarian algebra}
if it takes this form upon extension of scalars to some (and hence any) field $\widetilde{\FF}$ that splits $\LL$, i.e., 
$\LL\otimes\widetilde{\FF}\simeq\widetilde{\FF}\times\widetilde{\FF}\times\widetilde{\FF}$.
  
For any trialitarian algebra, the automorphism group scheme $\AAut_\FF(E,\LL,\sigma,\alpha)$ does not depend on the choice of $\rho$. 
The subgroupscheme of $\LL$-linear automorphisms is a twisted form of $\PGOs_8^+$ and is denoted by $\PGOs^+(E,\LL,\sigma,\alpha)$
(see \cite[\S 44.A]{KMRT}). We have a short exact sequence:
\begin{equation}\label{eq:exact_trialitarian}
\xymatrix{
\mathbf{1}\ar[r] & \PGOs^+(E,\LL,\sigma,\alpha)\ar[r] & \AAut_\FF(E,\LL,\sigma,\alpha)\ar[r]^-{\pi} & 
\AAut_\FF(\LL)\ar[r] & \mathbf{1}
}
\end{equation}
(compare with \eqref{eq:exact_D4}). If $\LL=\FF\times\KK$ where $\KK$ is a quadratic \'{etale} algebra over $\FF$ then 
$E\simeq A\times\Cl(A,\sigma_A)$ and the restriction to the first factor is a closed embedding $\PGOs^+(E,\LL,\sigma,\alpha)\to\AAut_\FF(A,\sigma_A)$ 
whose image is the connected component of $\AAut_\FF(A,\sigma_A)$.

\subsection{The trialitarian algebra $\End_\LL(V)$}

Let $(V,\LL,\rho,*,Q)$ be a cyclic composition of rank $8$. Consider $E=\End_\LL(V)$, with involution $\sigma$ determined by the quadratic form $Q$. 
Since $\LL$ is Galois, we have $\Delta=\FF\times\FF$ and we can identify $(E\otimes\Delta)^\rho$ with 
 $E^{\rho^2}\times E^{\rho}$.

\begin{df}\label{df:tri_alg}
The multiplication $*$ of $V$ allows us to define the remaining part of the structure of trialitarian algebra on $E$, namely, 
the isomorphism of $\LL$-algebras with involution
$
\alpha\colon\Cl(E,\sigma)\stackrel{\sim}{\to}E^{\rho^2}\times E^{\rho},
$
by using the Clifford algebra $\Cl(V,Q)$ as follows. Identities \eqref{eq:cyclic_comp_id} imply that the mapping
\[
x\mapsto\begin{pmatrix}0 & r_x \\ l_x & 0\end{pmatrix}\in\End_{\LL}\big(V^{\rho^2}\oplus V^{\rho}\big),\;x\in V,
\]
where $l_x(y)\bydef x*y\defby r_y(x)$, extends to an isomorphism of $\ZZ_2$-graded algebras with involution
\[
\alpha_V\colon(\Cl(V,Q),\tau)\stackrel{\sim}{\to}(\End_{\LL}\big(V^{\rho^2}\oplus V^{\rho}\big),\tilde\sigma),
\]
where $\tau$ is the standard involution of the Clifford algebra and $\tilde\sigma$ is induced by the quadratic form 
$Q^{\rho^2}\perp Q^{\rho}$ on $V^{\rho^2}\oplus V^{\rho}$, with $Q^\rho(x)\bydef\rho^{-1}(Q(x))$. 
Then $\alpha$ is obtained by restricting $\alpha_V$ to the even part $\Cl_0(V,Q)$ and identifying $\Cl(E,\sigma)$ with $\Cl_0(V,Q)$. 
Explicitly,
\[
\alpha\colon \kappa\bigl(xb_Q(y,\cdot)\bigr)\mapsto(r_x l_y, l_x r_y),\;x,y\in V.
\]
\end{df}

More generally, if $(V,\LL,\beta,Q)$ is a twisted composition where $\LL$ is not necessarily Galois then, for a fixed $3$-cycle $\rho$, one can consider 
the extension of scalars from $\FF$ to $\Delta$ and use the corresponding unique cyclic composition $(V\otimes\Delta,\LL\otimes\Delta,\rho,*,Q\otimes\id)$ 
to construct $\alpha$ that makes $E=\End_\LL(V)$ a trialitarian algebra (see \cite[Proposition 36.19]{KMRT}). 
The trialitarian algebras of this form are characterized by the property that the class $[E]$ in the Brauer group of $\LL$ is trivial 
(see \cite[Proposition 44.16]{KMRT}).

For $E=\End_\LL(V)$, let $\Gs$ be the group scheme of invertible $\FF$-linear endomorphisms of $V$ that are semilinear over $\LL$ 
and consider the homomorphism $\mathrm{Int}\colon\Gs\to\AAut_\FF(E)$ where $\mathrm{Int}_\cR(a)$ is the conjugation by $a\in\Gs(\cR)$ 
on $E_\cR=\End_{\LL\ot\cR}(V\ot\cR)$, for all $\cR\in\Alg_\FF$. 
Then $\mathrm{Int}$ restricts to a quotient map $\AAut_\FF(V,\LL,\beta,Q)\to \AAut_\FF(E,\LL,\sigma,\alpha)$. 
In fact, we obtain a morphism of short exact sequences from \eqref{eq:exact_twisted_comp} to \eqref{eq:exact_trialitarian},
which is the identity on $\AAut_\FF(\LL)$ and the quotient modulo the center on $\Spins(V,\LL,\beta,Q)$.

\subsection{Trialitarian algebras and simple Lie algebras of type $D_4$}\label{ss:trialitarian_D4}

As mentioned in the Introduction, there is a fundamental connection between trialitarian algebras and simple Lie algebras of type $D_4$ (see \cite[\S 45]{KMRT}). 
For any trialitarian algebra, the canonical map $\kappa$ yields an injective homomorphism 
$\frac12\kappa\colon\mathrm{Skew}(E,\sigma)\to\mathrm{Skew}(\Cl(E,\sigma),\ul{\sigma})$ 
of Lie algebras over $\LL$, and the $\FF$-subspace
\[
\cL(E,\LL,\rho,\sigma,\alpha)\bydef\{x\in\mathrm{Skew}(E,\sigma)\;|\;\alpha(\kappa(x))=2x\otimes 1\}
\]
is a simple Lie algebra of type $D_4$. For $E=\End_\LL(V)$, this Lie algebra is equal to $\Der_\LL(V,\beta,Q)$.  

The restriction from $E$ to $\cL(E)\bydef \cL(E,\LL,\rho,\sigma,\alpha)$ yields an isomorphism of affine group schemes
$\mathrm{Res}\colon\AAut_\FF(E,\LL,\sigma,\alpha)\stackrel{\sim}{\to}\AAut_\FF(\cL(E))$.
Note that, if $E=\End_\LL(V)$ then the composition $\mathrm{Res}\circ\mathrm{Int}$ is the adjoint representation 
$
\Ad\colon\AAut_\FF(V,\LL,\beta,Q)\to\AAut_\FF(\Der_\LL(V,\beta,Q)). 
$

%
%
\section{Gradings on Cayley algebras and simple Lie algebras of type $G_2$}\label{s:octonions_G2}

Gradings on the two Cayley algebras $\OO$ and $\OO_s$ over $\RR$ will be instrumental to classify Type III gradings on the simple Lie algebras of type $D_4$. Over algebraically closed fields, gradings on Cayley algebras have been classified up to isomorphism in \cite{Eld98} and \cite[\S 4.2]{EKmon}. In this section, this will be extended to arbitrary fields and, as a particular case, the situation over $\RR$ (or any real closed field) will be given explicitly.

Recall the \emph{Cayley--Dickson doubling process:} given an associative Hurwitz algebra $\cQ$ over a field $\FF$ with norm $n$ such that its polar form is nondegenerate (which excludes the ground field if its characteristic is $2$) and any nonzero scalar $\alpha\in\FF^\times$, let $\cC$ be the direct sum  of two copies of $\cQ$, where we formally write the element 
$(x,y)$ as $x+yu$, so $\cC=\cQ\oplus \cQ u$. Then $\cC$ is a Hurwitz algebra with multiplication and norm given by:
\[
\begin{split}
&(a+bu)(c+du)=(ac+\alpha\overline{d}b)+(da+b\overline{c})u,\\
&n\bigl(a+bu\bigr)=n(a)-\alpha n(b).
\end{split}
\]
The new algebra thus obtained will be denoted by $\CD(\cQ,\alpha)$. For example, the classical real algebras of complex numbers, quaternions, and octonions, are obtained as $\CC=\CD(\RR,-1)$, $\HH=\CD(\CC,-1)$, and $\OO=\CD(\HH,-1)$. We will write $\CD(\cA,\alpha,\beta)\bydef \CD\bigl(\CD(\cA,\alpha),\beta\bigr)$ and similarly for $\CD(\cA,\alpha,\beta,\gamma)$, so, for instance, we have $\OO=\CD(\RR,-1,-1,-1)$.

Conversely, given any Hurwitz algebra $\cC$ with norm $n$ and a subalgebra $\cQ$ such that the restriction to $\cQ$ of the polar form of $n$ is nondegenerate, and given any nonisotropic element $u\in \cQ^\perp$, it follows that $n\bigl(\cQ,\cQ u\bigr)=0$ and that $\cQ\oplus \cQ u$ is a subalgebra of $\cC$ isomorphic to $\CD(\cQ,\alpha)$ with $\alpha=-n(u)=u^2$.

Now let $\cQ$ be a quaternion subalgebra of a Cayley algebra $\cC $ over a field $\FF$, and let $u\in \cQ^\perp$ with $n(u)=-\gamma\neq 0$. Then $\cC =\cQ\oplus \cQ u$ is isomorphic to $\CD(\cQ ,\gamma)$, and this gives a $\ZZ_2$-grading on $\cC$ with $\cC_{\bar 0}=\cQ $ and $\cC_{\bar 1}=\cQ^\perp=\cQ u$.

The quaternion subalgebra $\cQ $ can be obtained in turn from a quadratic \'etale subalgebra $\cK$ as $\cQ =\cK\oplus \cK v\simeq\CD(\cK,\beta)$, with $v\in \cQ \cap \cK^\perp$, $n(v)=-\beta\neq 0$. Then $\cC $ is isomorphic to $\CD(\cK,\beta,\gamma)$, and this gives a $\ZZ_2^2$-grading on $\cC $ with $\cC _{(\bar 0,\bar 0)}=\cK$, $\cC _{(\bar 1,\bar 0)}=\cK v$, $\cC _{(\bar 0,\bar 1)}=\cK u$, $\cC _{(\bar 1,\bar 1)}=(\cK v)u=\cK(uv)$.

If $\chr\FF\neq 2$, then $\cK$ can be obtained in turn from the ground field $\FF$: $\cK=\FF\oplus\FF w\simeq \CD(\FF,\alpha)$, with $w\in \cK\cap\FF^\perp$, $n(w)=-\alpha\neq 0$, so $\cC =\CD(\FF,\alpha,\beta,\gamma)$, and this gives a $\ZZ_2^3$-grading on $\cC $.

These gradings by $\ZZ_2^r$, $r=1,2,3$, are called the \emph{gradings induced by the Cayley--Dickson doubling process} \cite[p.~131]{EKmon}. The groups $\ZZ_2^r$ are the universal grading groups.

In a different vein, on the split Cayley algebra $\cC _s$ over a field $\FF$, there is a fine grading, called \emph{Cartan grading}, whose homogeneous components are the eigenspaces for the action of a maximal torus of $\AAut_\FF(\cC _s)$. Its universal grading group is $\ZZ^2$.

\begin{theorem}[\cite{Eld98},{\cite[Theorem 4.12]{EKmon}}]
Any nontrivial grading on a Cayley algebra is, up to equivalence, either a grading induced by the Cayley--Dickson doubling process or a coarsening of the Cartan grading on the split Cayley algebra. \qed
\end{theorem}

Moreover, the proof of \cite[Theorem 4.12]{EKmon} shows that if $\cC =\bigoplus_{g\in G}\cC _g$ is a grading on a Cayley algebra $\cC $ and if $\cC _e$ is a split Hurwitz algebra of dimension $\geq 2$, then $\cC _e$ contains two nontrivial orthogonal idempotents, and the grading is, up to equivalence, a coarsening of the Cartan grading. Thus we may strengthen  the previous result as follows:

\begin{theorem}\label{th:CD_or_Cartan}
Any nontrivial grading on a Cayley algebra is, up to equivalence, either a grading induced by the Cayley--Dickson doubling process starting with a Hurwitz division subalgebra, or a coarsening of the Cartan grading on the split Cayley algebra. \qed
\end{theorem}

In order to state the classification of $G$-gradings, up to isomorphism, on Cayley algebras over an arbitrary field $\FF$, we need some notation. Let $\cC_s$ be the split Cayley algebra and consider its Cartan grading as described in \cite[p.~131]{EKmon}. There is a \emph{good basis} $\{e_1,e_2,u_1,u_2,u_3,v_1,v_2,v_3\}$ of $\cC _s$ with
\begin{gather*}
e_i^2=e_i,\quad e_1u_j=u_je_2=u_j,\quad e_2v_j=v_je_1=v_j,\\
u_ju_k=\veps_{jkl} v_l,\quad v_jv_k=\veps_{jkl} u_l,\\
u_jv_j=-e_1,\quad v_ju_j=-e_2,
\end{gather*}
for any $i=1,2$, $j,k,l=1,2,3$, and all other products equal to $0$, where $\veps_{jkl}$ is the Levi-Civita symbol ($\veps_{jkl}=0$ if two indices are equal, otherwise $\veps_{jkl}$ is the sign of the permutation $1\mapsto j$, $2\mapsto k$, $3\mapsto l$). The Cartan grading is then the $\ZZ^2$-grading on $\cC _s$ with $\deg(e_1)=\deg(e_2)=(0,0)$ and
\begin{align*}
\deg(u_1)&=(1,0)=-\deg(v_1),\\ 
\deg(u_2)&=(0,1)=-\deg(v_2),\\ 
-\deg(u_3)&=(1,1)=\deg(v_3).
\end{align*}

Consider the following gradings on Cayley algebras:
\begin{itemize}
\item Let $\gamma=(g_1,g_2,g_3)$ be a triple of elements in $G$ with $g_1g_2g_3=e$. Denote by $\Gamma_{\cC_s}(G,\gamma)$  the $G$-grading on $\cC_s$ induced from the Cartan grading by the homomorphism $\ZZ^2\rightarrow G$ sending $(1,0)\mapsto g_1$ and $(0,1)\mapsto g_2$. For two such triples $\gamma$ and $\gamma'$, we write $\gamma \sim\gamma'$ if there exists a permutation $\pi\in\textup{Sym}(3)$ such that either $g_i'=g_{\pi(i)}$ for all $i=1,2,3$, or $g_i'=g_{\pi(i)}^{-1}$ for all $i=1,2,3$. Note that if $\gamma=(e,e,e)$, we obtain the trivial grading.

\item Let $\cQ $ be a quaternion division subalgebra of $\cC$ and let $T=\langle t\rangle$ be a subgroup of order $2$ in $G$. Denote by $\Gamma_\cC(G,\cQ ,T)$ the $G$-grading on $\cC$ with $\cC_e=\cQ $ and $\cC_t=\cQ ^\perp$. Up to equivalence, this is a grading by $\ZZ_2$ induced by the Cayley--Dickson doubling process.

\item Let $\cK$ be a quadratic separable subfield of $\cC$ and let $T$ be an elementary abelian subgroup of order $4$ in $G$. Then $\cC$ is obtained by the Cayley--Dickson doubling process as $\cC=\CD(\cK,\alpha,\beta)=\cK\oplus \cK u\oplus \cK v\oplus \cK(uv)$ with $n(u)=-\alpha$, $n(v)=-\beta$. Let $t_1,t_2$ be generators of $T$ and consider the $G$-grading on $\cC$ with $\cC_e=\cK$, $\cC_{t_1}=\cK u$, $\cC_{t_2}=\cK v$, $\cC_{t_1t_2}=\cK(uv)$.

Unlike in the previous case, here we have some freedom in choosing $\alpha=-n(u)$ and $\beta=-n(v)$. More precisely, for any nonisotropic $x\in \cC$, define $\mu(x):=n(x)n(\cK^\times)$, which is an element in the quotient group $\FF^\times/n(\cK^\times)$. Then $\mu(x)$ is trivial for any $x\in \cK^\times$, while $\mu(x)\in -\alpha n(\cK^\times)$ for any $x\in \cK u$, $\mu(x)\in -\beta n(\cK^\times)$ for any $x\in \cK v$, and $\mu(x)\in \alpha\beta n(\cK^\perp)$ for any $x\in \cK(uv)$. Thus, we obtain a well-defined map $T\rightarrow \FF^\times/n(\cK^\times)$, which we also denote by $\mu$, such that $\mu(t)=\mu(x)$ for any $x\in \cC_t$, and this map is a group homomorphism because the norm is multiplicative.
The $G$-grading above depends, up to isomorphism, only on the subalgebra $\cK$, the subgroup $T$, and the homomorphism $\mu\colon T\rightarrow \FF^\times/n(\cK^\times)$; it will be denoted by $\Gamma_\cC(G,\cK,T,\mu)$.

Note that the homomorphism $\mu$ is not arbitrary: if $T=\{e,t_1,t_2,t_3\}$ and if $\tilde\mu_i\in \FF^\times$ denotes a representative of $\mu(t_i)$, then the norm $n$ must be isometric to the orthogonal sum $n_\cK\perp \tilde\mu_1 n_\cK\perp \tilde\mu_2 n_\cK\perp\tilde\mu_3 n_\cK$, where $n_\cK$ denotes the restriction of $n$ to $\cK$. Such homomorphisms will be called \emph{admissible}.

\item If $\chr\FF\neq 2$, $\cC$ can be obtained by the Cayley--Dickson doubling process as $\cC=\CD(\FF,\alpha,\beta,\gamma)=\FF\oplus \FF u\oplus\FF v\oplus \FF uv\oplus \FF w\oplus \FF uw\oplus \FF vw\oplus\FF (uv)w$, where $n(u)=-\alpha$, $n(v)=-\beta$ and $n(w)=-\gamma$. Let $T$ be an elementary abelian subgroup of order $8$ in $G$ and let $t_1,t_2,t_3$ be generators of $T$. Consider the $G$-grading determined by $\deg(u)=t_1$, $\deg(v)=t_2$, and $\deg(w)=t_3$. 

Similarly to the previous case, we obtain a group homomorphism $\mu\colon T\rightarrow \FF^\times/(\FF^\times)^2$ given by $\mu(t)=\mu(x)(\FF^\times)^2$ for any $x\in \cC_t$. The $G$-grading above depends, up to isomorphism, only on $T$ and $\mu$; it will be denoted by $\Gamma_\cC(G,\FF,T,\mu)$. 

Again, $\mu$ is not an arbitrary homomorphism: if $T=\{e,t_i:1\leq i\leq 7\}$ and if $\tilde\mu_i\in\FF^\times$ denotes a representative of $\mu(t_i)$, then the norm $n$ must be isometric to the diagonal quadratic form  $\diag(1,\tilde\mu_1,\ldots,\tilde\mu_7)$. Such homomorphisms will be called \emph{admissible}.
\end{itemize}

The next result extends \cite[Theorem 4.21]{EKmon} to arbitrary fields.

\begin{theorem}\label{th:gradings_Cayley}
Let $\cC$ be a Cayley algebra over a field $\FF$ and let $G$ be an abelian group. Then any nontrivial $G$-grading on $\cC$ is isomorphic to one of the following:
\begin{itemize}
\item $\Gamma_{\cC_s}(G,\gamma)$ for a triple $\gamma\in G^3$, $\gamma\neq (e,e,e)$, if $\cC$ is the split Cayley algebra;

\item $\Gamma_\cC(G,\cQ ,T)$ for a quaternion division subalgebra $\cQ $ of $\cC$ and a subgroup $T$ of order $2$ in $G$;

\item $\Gamma_\cC(G,\cK,T,\mu)$ for a quadratic separable subfield $\cK$ of $\cC$, an elementary abelian subgroup $T$ of order $4$ in $G$, and an admissible group homomorphism $\mu\colon T\rightarrow \FF^\times/n(\cK^\times)$;

\item $\Gamma_\cC(G,\FF,T,\mu)$ for an elementary abelian subgroup $T$ of order $8$ in $G$ and a admissible group homomorphism $\mu\colon T\rightarrow \FF^\times/(\FF^\times)^2$, if $\chr\FF\neq 2$.
\end{itemize}
Two $G$-gradings in different items are not isomorphic, and
\begin{itemize}
\item $\Gamma_{\cC_s}(G,\gamma)$ is isomorphic to $\Gamma_{\cC_s}(G,\gamma')$ if and only if $\gamma\sim\gamma'$;

\item $\Gamma_\cC(G,\cQ ,T)$ is isomorphic to $\Gamma_\cC(G,\cQ ',T')$ if and only if $\cQ $ is isomorphic to $\cQ '$ and $T=T'$;

\item $\Gamma_\cC(G,\cK,T,\mu)$ is isomorphic to $\Gamma_\cC(G,\cK',T',\mu')$ if and only if $\cK$ is isomorphic to $\cK'$ and  $(T,\mu)=(T',\mu')$;

\item $\Gamma_\cC(G,\FF,T,\mu)$ is isomorphic to $\Gamma_\cC(G,\FF,T',\mu')$ if and only $(T,\mu)=(T',\mu')$.
\end{itemize}
\end{theorem}

\begin{proof}
Any $G$-grading on $\cC$ is isomorphic to a grading on our list by Theorem~\ref{th:CD_or_Cartan}. The isomorphism condition for coarsenings of the Cartan grading is proved as in \cite[Theorem 4.21]{EKmon}. In the remaining cases, any isomorphism of $G$-gradings $\Gamma$ and $\Gamma'$ takes the homogeneous component $\cC_e$ of $\Gamma$ to the component $\cC'_e$ of $\Gamma'$. This shows that gradings in different items cannot be isomorphic and that the subalgebra $\cQ$ (respectively $\cK$) must be isomorphic to $\cQ'$ (respectively $\cK'$) in the corresponding items. Also, any isomorphism of gradings preserves the support and the norm, which gives the conditions $(T,\mu)=(T',\mu')$. Conversely, assume for instance that $\cK$ and $\cK'$ are isomorphic quadratic separable subfields of $\cC$, and that we have two $G$-gradings $\Gamma=\Gamma_\cC(G,\cK,T,\mu)$ and $\Gamma'=\Gamma_\cC(G,\cK',T,\mu)$. Then we have $\Gamma:\cC=\cK\oplus \cK u\oplus \cK v\oplus \cK(uv)$ and $\Gamma':\cC=\cK'\oplus \cK'u'\oplus \cK'v'\oplus \cK'(u'v')$. As $\mu(u)=\mu(u')$, there exists $a\in (\cK')^\times$ such that $n(u)=n(a)n(u')=n(au')$, so any isomorphism $\varphi\colon\cK\rightarrow \cK'$ extends to an isomorphism $\varphi_1\colon\CD(\cK,\alpha)=\cK\oplus \cK u\rightarrow \CD(\cK',\alpha')=\cK'\oplus \cK'u'$ with $\varphi_1(u)=au'$. Also, there exists $b\in (\cK')^\times$ such that $n(v)=n(bv')$, so $\varphi_1$ extends to an automorphism $\varphi_2$ of $\cC$ with $\varphi_2(v)=bv'$. Then $\varphi_2$ is an isomorphism from $\Gamma$ to $\Gamma'$. The remaining cases are similar.
\end{proof}

If $\chr\FF\neq 2$, any grading $\Gamma_\cC(G,\cQ ,T)$ or $\Gamma_\cC(G,\cK,T,\mu)$ is a coarsening of a grading $\Gamma_\cC(G,\FF,T,\mu)$, which is equivalent to $\Gamma_\cC(\ZZ_2^3,\FF,\ZZ_2^3,\mu)$. On the other hand, if $\chr\FF=2$, any grading $\Gamma_\cC(G,\cQ ,T)$ is a coarsening of a grading $\Gamma_\cC(G,\cK,T,\mu)$. The following result is now clear.

\begin{corollary}\label{co:fine_gradings_Cayley}
Let $\cC$ be a Cayley algebra over a field $\FF$. Then, up to equivalence, the fine abelian group gradings on $\cC$ and their universal grading groups are the following.
\begin{itemize}
\item If $\cC$ is split, the Cartan grading with universal group $\ZZ^2$.
\item If $\chr\FF\neq 2$, the gradings $\Gamma_\cC(\ZZ_2^3,\FF,\ZZ_2^3,\mu)$ with universal group $\ZZ_2^3$; two such gradings, corresponding to $\mu$ and $\mu'$, are equivalent if and only if there is an automorphism $\varphi\in\Aut(\ZZ_2^3)$ such that $\mu'=\mu\circ\varphi$.
\item If $\chr\FF= 2$, the gradings $\Gamma_\cC(\ZZ_2^2,\cK,\ZZ_2^2,\mu)$ with universal group $\ZZ_2^2$, where $\cK$ is a quadratic separable subfield of $\cC$; two such gradings, corresponding to $(\cK,\mu)$ and $(\cK',\mu')$, are equivalent if and only if $\cK$ and $\cK'$ are isomorphic and there is an automorphism $\varphi\in\Aut(\ZZ_2^2)$ such that $\mu'=\mu\circ\varphi$. \qed
\end{itemize}
\end{corollary}

\begin{remark} 
We leave the classification of gradings up to isomorphism and of fine grading up to equivalence on quaternion algebras as an exercise for the reader. The situation is similar to Theorem \ref{th:gradings_Cayley} and Corollary \ref{co:fine_gradings_Cayley}, but simpler.
\end{remark}

If the characteristic of $\FF$ is different from $2$ and $3$, the Lie algebra of derivations $\Der_\FF(\cC)$ of a Cayley algebra $\cC$ is a simple Lie algebra of type $G_2$, and any simple Lie algebra of type $G_2$ appears in this way. (If $\chr\FF=3$, $\Der_\FF(\cC)$ is not simple, while if $\chr\FF=2$, $\Der_\FF(\cC)$ is isomorphic to $\frpsl_4(\FF)$, regardless of  the isomorphism class of $\cC$ \cite{CastilloElduque}.)  Moreover, assuming  $\chr\FF\neq 2$, the adjoint representation  gives an isomorphism of group schemes $\Ad\colon\AAut_\FF(\cC)\rightarrow \AAut\bigl(\Der_\FF(\cC)\bigr)$ (see \cite[Theorem 4.35]{EKmon} and \cite[Section 3]{CastilloElduque}), hence any $G$-grading on $\Der_\FF(\cC)$ is induced by a unique grading on $\cC$, with two gradings on $\cC$ being isomorphic (respectively, two fine gradings being equivalent) if and only if the induced gradings on $\Der_\FF(\cC)$ are isomorphic (respectively, equivalent). Therefore, Theorem \ref{th:gradings_Cayley} and Corollary \ref{co:fine_gradings_Cayley} immediately give the clasification, respectively, of $G$-gradings up to isomorphism and of fine gradings up to equivalence on the simple Lie algebras of type $G_2$.

\medskip

Over $\RR$, there is a unique quadratic field $\CC$ and a unique quaternion division algebra $\HH$. Moreover $n(\CC^\times)=n(\HH^\times)=n(\OO^\times)=(\RR^\times)^2=\RR_{>0}$, and $\RR^\times/(\RR^\times)^2$ is isomorphic to the cyclic group $\{\pm 1\}$.

It follows that, for the octonion division algebra $\OO$, the admissible group homomorphisms $\mu$ are all trivial, as $n(\OO^\times)=(\RR^\times)^2$, and we may write simply $\Gamma_\OO(G,\HH,T)$, $\Gamma_\OO(G,\CC,T)$ and $\Gamma_\OO(G,\RR,T)$, forgetting the $\mu$, where $T$ is an elementary abelian subgroup of order $2$, $4$ or $8$, respectively, in $G$. On the other hand, for the split octonion algebra $\OO_s$, no admissible group homomorphism $\mu\colon T\rightarrow \{\pm 1\}\simeq\RR^\times/(\RR^\times)^2$ can be trivial, i.e., it must take the value $-1$ (as otherwise the norm of $\OO_s$ would be definite), so there are three possibilities if the order of $T$ is $4$ and seven possibilities if the order is $8$. The next result is a consequence of Theorem \ref{th:gradings_Cayley}.

\begin{corollary}\label{co:real_Cayley}
Let $G$ be an abelian group.
\begin{enumerate}
\item Any nontrivial $G$-grading on the real octonion division algebra $\OO$ is isomorphic to exactly one of the following gradings:
\begin{itemize} 
\item $\Gamma_\OO(G,\HH,T)$ for a subgroup $T$ of order $2$ in $G$;
\item $\Gamma_\OO(G,\CC,T)$ for an elementary abelian subgroup $T$ of order $4$ in $G$;
\item $\Gamma_\OO(G,\RR,T)$ for an elementary abelian subgroup $T$ of order $8$ in $G$.
\end{itemize}
\item Any $G$-grading on the real split octonion algebra $\OO_s$ is isomorphic to one of the following gradings:
\begin{itemize}
\item $\Gamma_{\OO_s}(G,\gamma)$ for a triple $\gamma\in G^3$, with two such gradings, corresponding to $\gamma$ and $\gamma'$, being isomorphic if and only if $\gamma\sim\gamma'$;
\item $\Gamma_{\OO_s}(G,\HH,T)$ for a subgroup $T$ of order $2$ in $G$;
\item $\Gamma_{\OO_s}(G,\CC,T,\mu)$ for an elementary abelian subgroup $T$ of order $4$ in $G$ and a nontrivial group homomorphism $\mu\colon T\rightarrow \{\pm 1\}$ (so, for each such $T$, there are three different possibilities for $\mu$, which give nonisomorphic gradings);
\item $\Gamma_{\OO_s}(G,\RR,T,\mu)$ for an elementary abelian subgroup $T$ of order $8$ in $G$ and a nontrivial group homomorphism $\mu\colon T\rightarrow \{\pm1 \}$ (so, for each such $T$, there are seven different possibilities for $\mu$, which give nonisomorphic gradings).\qed
\end{itemize}
\end{enumerate}
\end{corollary}

As for fine gradings, Corollary \ref{co:fine_gradings_Cayley} gives the next result.

\begin{corollary}\label{co:real_fine_Cayley}
\null\quad
\begin{enumerate}
\item Up to equivalence, the only fine grading on the real octonion division algebra $\OO$ is $\Gamma_{\OO}(\ZZ_2^3,\RR,\ZZ_2^3)$.
\item Up to equivalence, the fine gradings on the real split octonion algebra $\OO_s$ are the Cartan grading and the grading $\Gamma_{\OO_s}(\ZZ_2^3,\RR,\ZZ_2^3,\mu)$, where $\mu\bigl((\bar 1,\bar 0,\bar 0)\bigr)=\mu\bigl((\bar 0,\bar 1,\bar 0)\bigr)=1$, $\mu\bigl((\bar 0,\bar 0,\bar 1)\bigr)=-1$. \qed
\end{enumerate}
\end{corollary}

\medskip 

Passing to the Lie algebras of derivations, we see (cf. \cite{g2f4}) that the compact real form of type $G_2$ admits, up to equivalence, only one fine grading, whose universal group is $\ZZ_2^3$ and which has $7$ homogeneous components of dimension $2$, while the split real form of type $G_2$ admits, up to equivalence, two fine gradings: the Cartan grading and a $\ZZ_2^3$-grading  analogous to the one on the compact form.

\section{Gradings of Type I and II on simple real Lie algebras of type $D_4$}\label{s:Types_I_II}

We will now briefly discuss gradings of Types I and II on the simple Lie algebras of type $D_4$ over the field $\RR$ (or any real closed field) 
by an abelian group $G$.
So, let $\cL=\frso_{p,q}(\RR)$, $p+q=8$, $p\ge q$. As mentioned in the Introduction, $\cL\simeq\cL(E)$ for the trialitarian algebra $E=A\times\Cl(A,\sigma_A)$
where $A=M_8(\RR)$ and $\sigma_A$ is an orthogonal involution of inertia $(p,q)$. The involution $\sigma$ of $E$ is $\sigma_A\times\ul{\sigma}_A$, the center 
$\LL$ is $\RR\times\KK$, where $\KK$ is the center of the Clifford algebra, and 
the trialitarian structure $\alpha$ is given by \cite[Proposition 43.15]{KMRT}.

\subsection{Classification up to isomorphism}

Let $\Gamma$ be a $G$-grading on $\cL$ and let $\eta=\eta_\Gamma$ be the corresponding homomorphism $G^D\to\AAut_\RR(\cL)\simeq\AAut_\RR(E,\LL,\sigma,\alpha)$.
In this section, we assume that $\Gamma$ is of Type~I or Type~II, i.e., the image $\pi\eta(G^D)$ in $\AAut_\RR(\LL)$ has order $1$ or $2$, respectively.
(Recall that this is automatic unless $p$ and $q$ are odd.) We will show how $\Gamma$ can be described in terms of a suitable matrix model, i.e., 
a central simple associative $\RR$-algebra (of degree $8$) $R$ equipped with an orthogonal involution $\varphi$. 
The $G$-gradings on such algebras with involution $(R,\varphi)$ are classified up to isomorphism in \cite{BKR_Lie}. 
We will say that a grading on $(R,\varphi)$ is of Type~I if the image of the corresponding homomorphism $G^D\to\AAut_\RR(R,\varphi)$ is contained in the 
connected component of $\AAut_\RR(R,\varphi)$, and of Type~II otherwise. The distinction can be made by looking at the image of the character group
$\widehat{G}\bydef\Hom(G,\CC^\times)$ (i.e., the group of $\CC$-points of $G^D$) under the homomorphism $\widehat{G}\to\PGO_8(\CC)$ associated to the 
grading on the complexification of $(R,\varphi)$ --- see \cite[Lemma 33]{EK14}.

If $\Gamma$ is of Type~I then the corresponding grading on $\LL$ is trivial and the image $\eta(G^D)$ lies in $\PGOs^+_{p,q}$, 
which is the connected component of $\PGOs_{p,q}=\AAut_\RR(A,\sigma_A)$. 
Hence, $\Gamma$ can be identified with the restriction of a unique $G$-grading on $A=M_8(\RR)$ (of Type~I) 
to the Lie algebra of skew-symmetric elements $\frso_{p,q}(\RR)$. 
If $\Gamma$ is of Type~II then the grading on $\LL$ is nontrivial, but we still have a factor of $\LL$, and hence of $E$, that is $G$-graded.
The details vary depending on $(p,q)$, as follows.

If $(p,q)=(8,0)$ or $(4,4)$ then we can identify $\cL$ with the triality Lie algebra $\tri(\wb{\cC},\bullet,n)$ 
where $(\wb{\cC},\bullet,n)$ is the para-Hurwitz algebra associated, respectively, to $\cC=\OO$ and $\cC=\OO_s$. 
Here $\LL=\RR\times\RR\times\RR$ and $E=A\times B\times C$ where $A=B=C=M_8(\RR)$.
Consequently, one can define an outer $S_3$-action on $\cL$, which shows that the groups of $\RR$-points of \eqref{eq:exact_D4} form a split exact sequence:
\begin{equation}\label{eq:exact_80_44}
\xymatrix{
1\ar[r] & \PGO^+_{p,q}(\RR)\ar[r] & \Aut_\RR(E,\LL,\sigma,\alpha)\ar[r]^-{\pi_\RR} & S_3\ar[r] & 1
}
\end{equation}
If $\Gamma$ is of Type~II then $\pi\eta(G^D)$ is one of the three subgroupschemes of $\mathbf{S}_3$ corresponding to transpositions. 
Applying, if necessary, an automorphism of $E$ that projects onto a suitable $3$-cycle in $S_3$, we may assume that the transposition is $(2,3)$, 
i.e., $\eta(G^D)$ lies in $\AAut_\RR(A,\sigma_A)$, and $\Gamma$ can be identified with the restriction of a unique $G$-grading on $A$, 
just as for Type~I gradings. But note that, unlike in the case of Type~I, here the factor $A$ is distinguished. 
Therefore, in view of the exact sequence \eqref{eq:exact_80_44}, the classification of $G$-gradings on the Lie algebra $\cL=\frso_{8,0}(\RR)$, 
respectively $\cL=\frso_{4,4}(\RR)$,
has the following relationship with the classification of $G$-gradings on the matrix algebra 
$M_8(\RR)$ with involution given by a quadratic form on $\RR^8$ of inertia $(8,0)$,
respectively $(4,4)$:
\begin{itemize} 
\item one isomorphism class of Type~I gradings on $\cL$ may correspond to $1$, $2$ or $3$ isomorphism classes of Type~I gradings on $M_8(\RR)$ 
--- see the next subsection for details;
\item  
there is a bijection between the isomorphism classes of Type~II gradings on $\cL$ and those of Type~II gradings on $M_8(\RR)$.
\end{itemize}

If $(p,q)=(6,2)$ then $\LL=\RR\times\RR\times\RR$ and $E=A\times B\times C$ where $A=M_8(\RR)$ but $B=C=M_4(\HH)$, 
so the sequence of $\RR$-points of \eqref{eq:exact_D4} fails to be exact. Actually, we have the following split exact sequence:
\[
\xymatrix{
1\ar[r] & \PGO^+_{6,2}(\RR)\ar[r] & \Aut_\RR(E,\LL,\sigma,\alpha)\ar[r]^-{\pi_\RR} & \langle(2,3)\rangle\ar[r] & 1
}
\]
because any reflection in $\RR^8$ induces an automorphism of its even Clifford algebra that interchanges the factors $B$ and $C$.
If $\Gamma$ is of Type~I then it is the restriction of a unique grading on the distinguished factor $A$. 
If $\Gamma$ is a Type~II grading then we have two possibilities. 
If $\pi\eta(G^D)$ is the subgroupscheme of $\mathbf{S}_3$ corresponding to $(2,3)$ then $\eta(G^D)$ stabilizes 
the factor $A$ and hence, again, $\Gamma$ is the restriction of a unique grading on $A$. 
Otherwise, $\eta(G^D)$ stabilizes exactly one of the two factors $B$ and $C$. 
Applying, if necessary, an automorphism of $E$, we may assume that it is $B$.
Note that in this case the factor $B$ is distinguished, and also the group $\Aut_\RR(B,\sigma_B)$ is isomorphic to $\PGO^+_{6,2}(\RR)$.
Therefore, the set of isomorphism classes of gradings on the Lie algebra $\cL=\frso_{6,2}(\RR)$ is in bijection with the disjoint union of the 
following two sets:
\begin{itemize} 
\item the isomorphism classes of gradings on $M_8(\RR)$ with involution given by a quadratic form on $\RR^8$ of inertia $(6,2)$;
\item the isomorphism classes of Type~II gradings on $M_4(\HH)$ with involution given by a skew-hermitian form on $\HH^4$.
\end{itemize}

If $(p,q)=(7,1)$ or $(5,3)$ then $\LL=\RR\times\CC$, $E=M_8(\RR)\times M_8(\CC)$, and the groups of $\RR$-points of \eqref{eq:exact_D4} 
form the following split exact sequence:
\[
\xymatrix{
1\ar[r] & \PGO^+_{p,q}(\RR)\ar[r] & \Aut_\RR(E,\LL,\sigma,\alpha)\ar[r]^-{\pi_\RR} & \Aut_\RR(\CC)\ar[r] & 1
}
\]
Moreover, $\AAut_\RR(\CC)\simeq\mathbf{C}_2$ is the only subgroupscheme of order $2$ in $\AAut_\RR(\LL)\simeq\boldsymbol{\mu}_3\rtimes\mathbf{C}_2$, 
since there is only one subgroup of order $2$ in $S_3$ that is stable under the action of the absolute Galois group. 
Therefore, the isomorphism classes of $G$-gradings of Type~I or II on the Lie algebra $\cL=\frso_{7,1}(\RR)$, respectively $\cL=\frso_{5,3}(\RR)$,
are in bijection with the isomorphism classes of $G$-gradings on $M_8(\RR)$ with involution given by a quadratic form on $\RR^8$ of inertia $(7,1)$, 
respectively $(5,3)$.

\subsection{Related triples of gradings on $\End_\FF(\cC)$}

In this subsection, the ground field $\FF$ is arbitrary with $\chr{\FF}\ne 2$. 
We follow \cite[\S 3.1]{EK15}, but without assuming $\FF$ algebraically closed.
Let $\LL=\FF\times\FF\times\FF$ and let $(\wb{\cC},\bullet,n)$ be the 
para-Hurwitz algebra associated to a Cayley algebra $\cC$. Consider the cyclic composition $V=\wb{\cC}\otimes\LL$, with 
$\rho(\ell_1,\ell_2,\ell_3)=(\ell_2,\ell_3,\ell_1)$, and the corresponding trialitarian algebra $E=\End_\LL(V)$. 

Suppose we have a $G$-grading $\Gamma$ on $(E,\LL,\rho,\sigma,\alpha)$, i.e., a grading on the $\FF$-algebra $E$ such that 
the maps $\sigma$ and $\alpha$ preserve the degree. Assume that $\Gamma$ is of Type~I, i.e, its restriction to the center $\LL$ is trivial. 
Then the decomposition $\LL=\FF\times\FF\times\FF$ yields the $G$-graded decomposition $E=E_1\times E_2\times E_3$. 
Denote the $G$-grading on $E_i$ by $\Gamma_i$, $i=1,2,3$. 
We can identify each $(E_i,\sigma_i)$ with $\End_\FF(\cC)$ where the involution is induced by the norm $n$. 
Thus, we can view all the $\Gamma_i$ as gradings on the same algebra $\End_\FF(\cC)\simeq M_8(\FF)$. 
We will say that $(\Gamma_1,\Gamma_2,\Gamma_3)$ is the {\em related triple of gradings} associated to $\Gamma$. 

The grading $\Gamma$ corresponds to a homomorphism $\eta\colon G^D\to\AAut_\LL(E,\LL,\sigma,\alpha)$. 
Since each restriction homomorphism $\pi_i\colon\AAut_\LL(E,\LL,\sigma,\alpha)\to\AAut_\FF(E_i,\sigma_i)$ is a closed embedding 
whose image is the connected component of $\AAut_\FF(E_i,\sigma_i)$ (isomorphic to $\PGOs^+(\cC,n)$), 
each of the gradings $\Gamma_i$ (corresponding to $\pi_i\circ\eta$) is of Type~I and uniquely determines $\Gamma$.
Moreover, there exists $\Gamma$ such that $\Gamma_1$ is any given Type~I grading on the algebra with involution $\End_\FF(\cC)$. 

The outer action of $S_3$ on $\TRIs(\wb{\cC},\bullet,n)$ and on its quotient $\AAut_\LL(E,\LL,\sigma,\alpha)$ yields the following action 
on related triples of gradings: $A_3$ permutes the components of $(\Gamma_1,\Gamma_2,\Gamma_3)$ cyclically and the transposition $(2,3)$ sends 
$(\Gamma_1,\Gamma_2,\Gamma_3)$ to $(\wb{\Gamma}_1,\wb{\Gamma}_3,\wb{\Gamma}_2)$ where $\wb{\Gamma}_i$ denotes the image of $\Gamma_i$ 
under the inner automorphism of $\End_\FF(\cC)$ corresponding to the standard involution of $\cC$ (which is an improper isometry of $n$).

Let us see to what extent the knowledge of Type~I gradings on the algebra with involution $\End_\FF(\cC)$ 
allows us to classify Type~I gradings on the trialitarian algebra $E$ or, equivalently, on the Lie algebra $\cL(E)=\tri(\wb{\cC},\bullet,n)$. 
For two Type~I gradings $\Gamma'$ and $\Gamma''$ on $\End_\FF(\cC)$, we will write $\Gamma'\sim\Gamma''$ if there exists an element of $\PGO^+(\cC,n)$ 
sending $\Gamma'$ to $\Gamma''$. Thus, $\Gamma'$ and $\Gamma''$ are isomorphic if and only if $\Gamma'\sim\Gamma''$ or $\wb{\Gamma'}\sim\Gamma''$.

The stabilizer of a given Type~I grading $\Gamma$ on $E$ under the outer action of $S_3$ can have size $1$, $2$, $3$ or $6$. 
Therefore, we have the following three possibilities (the last one corresponding to sizes $3$ and $6$):
\begin{itemize}
\item If $\Gamma_i\nsim\Gamma_j$ for $i\ne j$ and $\Gamma_i\nsim\wb{\Gamma}_j$ for all $i,j$ then the isomorphism class of $\Gamma$ corresponds to $3$ distinct isomorphism classes of Type~I gradings on $\End_\FF(\cC)$ (one for each of $\Gamma_i$);
\item If $\Gamma_i\sim\wb{\Gamma}_i$ for some $i$ and $\Gamma_i\nsim\Gamma_j\sim\wb{\Gamma}_k$, where $\{i,j,k\}=\{1,2,3\}$, then the isomorphism class of $\Gamma$ corresponds to $2$ distinct isomorphism classes of Type~I gradings on $\End_\FF(\cC)$ (one for $\Gamma_i$ and one for $\Gamma_j$);
\item If $\Gamma_1\sim\Gamma_2\sim\Gamma_3$ then the isomorphism class of $\Gamma$ corresponds to $1$ isomorphism class of Type~I gradings on $\End_\FF(\cC)$.  
\end{itemize}
Given $\Gamma_1$, one can --- in principle --- compute $\Gamma_2$ and $\Gamma_3$, but obtaining explicit formulas seems to be difficult.

\section{Lifting Type III gradings from $\End_\LL(V)$ to $V$}\label{s:lifting}

Let $\LL$ be a cubic \'{e}tale algebra over a a field $\FF$, $\chr\FF\ne 2,3$, let $(V,\LL,\beta,Q)$ be a twisted composition of rank $8$, 
and let $E=\End_\LL(V)$ be the corresponding trialitarian algebra. 
Suppose we have a Type~III grading $\Gamma$ on $(E,\LL,\rho,\sigma,\alpha)$ by an abelian group $G$. Recall that this means that the projection of the image of 
$\eta=\eta_\Gamma\colon G^D\to\AAut_\FF(E,\LL,\sigma,\alpha)$ in $\AAut_\FF(\LL)$ is $\mathbf{A}_3$. 
We consider the problem of ``lifting'' from $E$ to $V$, i.e., the existence of a homomorphism $\eta'$ that would complete the following commutative diagram:
\begin{equation}
\begin{aligned}\label{eq:eta_diag}
\xymatrix{
& \AAut_\FF(V,\LL,\beta,Q)\ar[dd]^{\mathrm{Int}}\\
G^D\ar@{-->}[ru]^-{\eta'}\ar[rd]^-{\eta} &\\
& \AAut_\FF(E,\LL,\sigma,\alpha)
}
\end{aligned}
\end{equation}
or, equivalently, the existence of a $G$-grading $\Gamma'$ on $(V,\LL,\beta,Q)$, i.e., a grading on $V$ as an $\LL$-module 
making the linearizations of the quadratic maps $\beta$ and $Q$ degree-preserving, such that $\Gamma'$ would induce the given grading $\Gamma$ on $E=\End_\LL(V)$. 

If $\FF$ is algebraically closed then $\LL=\FF\times\FF\times\FF$ and, for a fixed $3$-cycle $\rho$ in $A_3$, $V$ admits a unique multiplication $*$ that makes it a cyclic composition. Note that
$\AAut_\FF(V,\LL,\rho,*,Q)$ is the inverse image of $\mathbf{A}_3$ under the projection $\AAut_\FF(V,\LL,\beta,Q)\to\AAut_\FF(\LL)\simeq\mathbf{S}_3$
(in other words, the stabilizer of $\rho$). Thus, the above lifting problem for $(V,\LL,\beta,Q)$ is equivalent to the same problem for $(V,\LL,\rho,*,Q)$, 
which was solved in our paper \cite{EK15}. We will first recall the relevant results from that paper and give them a homological interpretation. 
Then we will consider the case of a real closed field $\FF$ (such as $\RR$) by extending scalars to its algebraic closure and applying Galois descent.

\subsection{Lifting over algebraically closed fields}

Assume $\FF$ is algebraically closed. Then we have $\LL=\FF\times\FF\times\FF$, $V\simeq\wb{\cC}\otimes\LL$, 
where $(\wb{\cC},\bullet,n)$ is the para-Hurwitz algebra associated to the unique Cayley algebra $\cC$, 
and it is shown in \cite[\S 4.2]{EK15} (leading up to Theorem 12) that, for some $\lambda\in\LL^\times$, 
there exists a grading $V=\bigoplus_{g\in G}V_g$ on the cyclic composition $(V,\LL,\rho,\lambda *,\lambda^\sharp Q)$ 
that induces $\Gamma$ on $E=\End_\LL(V)$. 
(Note that the similar cyclic compositions $(V,\LL,\rho,\lambda *,\lambda^\sharp Q)$ and $(V,\LL,\rho,*,Q)$ 
induce the same trialitarian structure on $E$.)
Since $\FF$ is algebraically closed, we can find $\ell\in\LL^\times$ such that $\ell^{-1}\ell^\sharp =\lambda$ 
(see the proof of Lemma \ref{lm:FxK} below) and hence $V=\bigoplus_{g\in G}\ell V_g$ 
is a grading on the original cyclic composition $(V,\LL,\rho,*,Q)$ 
that induces $\Gamma$ on $E$. This proves the existence of $\eta'$.
We will now discuss the extent to which $\eta'$ is not unique, which will be crucial for Galois descent.

First suppose we have a short exact sequence of groups $1\to A\to B\to C\to 1$ and a group homomorphism 
$\phi\colon H\to C$ that lifts to a homomorphism $\phi'\colon H\to B$. It is well known and easy to check that 
all possible liftings of $\phi$ have the form $\xi\phi'$ where $\xi\colon H\to A$ is a (left) $1$-cocycle:
$\xi(xy)=\xi(x)(x\cdot \xi(y))$ for all $x,y\in H$, with $H$ acting on $A$ through the composition of $\phi'$
with the inner action of $B$ on its normal subgroup $A$. Thus, the fiber of the map $\Hom(H,B)\to\Hom(H,C)$ over 
$\phi$ is in bijection with the pointed set $\mathrm{Z}^1(H,A)$ of all $1$-cocycles, 
which depends on the chosen ``base point'' $\phi'$ in the fiber. 

We will use an analog of this result for affine group schemes over $\FF$. 
Recall that, for affine group schemes $\mathbf{H}$ and $\mathbf{A}$, a left action of $\mathbf{H}$ on $\mathbf{A}$ 
is a morphism of affine schemes $\theta\colon\mathbf{H}\times\mathbf{A}\to\mathbf{A}$ such that, 
for all $\cR$ in $\Alg_\FF$, $\theta_\cR\colon\mathbf{H}(\cR)\times\mathbf{A}(\cR)\to\mathbf{A}(\cR)$ is an 
action of $\mathbf{H}(\cR)$ on the group $\mathbf{A}(\cR)$ by automorphisms. Then we can define  
$\mathrm{Z}^1(\mathbf{H},\mathbf{A})$ as the set of all morphisms of affine schemes $\xi\colon\mathbf{H}\to\mathbf{A}$
such that, for all $\cR$ in $\Alg_\FF$, the map $\xi_\cR\colon\mathbf{H}(\cR)\to\mathbf{A}(\cR)$ is a $1$-cocycle.
The distinguished point of $\mathrm{Z}^1(\mathbf{H},\mathbf{A})$ is given by mapping each $\mathbf{H}(\cR)$ to the 
identity element of $\mathbf{A}(\cR)$. We can also define the group $\mathrm{Z}^0(\mathbf{H},\mathbf{A})$ as the 
fixed points of $\mathbf{H}$ in the group $\mathbf{A}(\FF)$, i.e., the elements $a\in\mathbf{A}(\FF)$ satisfying
$x\cdot a=a$ for all $x\in\mathbf{H}(\cR)$ and all $\cR$ in $\Alg_\FF$. (Note that $\mathbf{A}(\FF)$ is regarded as 
a subgroup of $\mathbf{A}(\cR)$ through the canonical map $\FF\to\cR$.)

If $\mathbf{A}$ is abelian then we can define a complex of abelian groups 
$\mathrm{C}^k(\mathbf{H},\mathbf{A}) \bydef \mathrm{Aff}_\FF(\mathbf{H}^k,\mathbf{A})$ for all $k=0,1,2,\ldots$,
with the convention $\mathbf{H}^0\bydef\mathbf{1}$ so that $\mathrm{C}^0(\mathbf{H},\mathbf{A})=\mathbf{A}(\FF)$, 
where the differentials are given by the usual formulas 
of group cohomology for each $\cR$ in $\Alg_\FF$. For example, 
$\mathrm{d}\colon\mathbf{A}(\FF)\to\mathrm{Aff}_\FF(\mathbf{H},\mathbf{A})$ is given (in multiplicative notation) 
by $(\mathrm{d}a)(x)=(x\cdot a)a^{-1}$ for all $x\in\mathbf{H}(\cR)$ and all $\cR$ in $\Alg_\FF$.
This leads to cohomology groups $\mathrm{H}^k(\mathbf{H},\mathbf{A})$ for all $k\ge 0$.

Let us now apply these considerations to $\mathbf{H}=G^D$ and the short exact sequence of affine group schemes 
given by the quotient map $\mathrm{Int}$ from Diagram \eqref{eq:eta_diag}, i.e.,
\[
\xymatrix{
\mathbf{1}\ar[r] & \mathbf{A}\ar[r] & \TRIs(\wb{\cC},\bullet,n)\rtimes\mathbf{S}_3\ar[r]^{\mathrm{Int}} & \AAut_\FF(E,\LL,\sigma,\alpha)\ar[r] & \mathbf{1}
}
\]
where $\mathbf{A}\simeq\boldsymbol{\mu}_2\times\boldsymbol{\mu}_2$ is the center of $\TRIs(\wb{\cC},\bullet,n)\simeq\Spins(\cC,n)$,
given by
\[
\mathbf{A}(\cR)=\{(\lambda_1,\lambda_2,\lambda_3)\in\cR\times\cR\times\cR\;|\;\lambda_i^2=1\mbox{ and }\lambda_1\lambda_2\lambda_3=1\}
\]
for all $\cR$ in $\Alg_\FF$. By \cite[Corollary 14]{EK15}, there are exactly four gradings on $(V,\LL,\rho,*,Q)$ that 
induce $\Gamma$ on $E$, and if $\Gamma':V=\bigoplus_{g\in G}V_g$ is one of them, then they all have the form $\Gamma'_a:V=\bigoplus_{g\in G}aV_g$ 
where $a$ ranges over the four elements of the group $\mathbf{A}(\FF)=Z(\TRI(\wb{\cC},\bullet,n))\subset\LL^\times$. 
Let us compute the homomorphism $\eta''$ corresponding to the grading $\Gamma'_a$ in terms of the homomorphism $\eta'$ corresponding to $\Gamma'$.
Write $a=\sum_{h\in G}a_h$ where $a_h\in\LL_h$. Then, for any $v\in V_g$, $r\in\cR$ and $f\in G^D(\cR)=\Alg_\FF(\FF G,\cR)$, 
Equation \eqref{eq:def_eta} gives us the following:
\[
\begin{split}
\eta'_\cR(f)(av\otimes r)&=\sum_{h\in G}\eta'_\cR(f)(a_h v\otimes r)=\sum_{h\in G}a_h v\otimes f(hg)r\\
&=\Big(\sum_{h\in G}a_h\otimes f(h)\Big)(a^{-1}\otimes 1)(av\otimes f(g)r)=(f\cdot a)a^{-1}\eta''_\cR(f)(av\otimes r),
\end{split}
\]
where the action of $G^D(\cR)$ on $(\LL\otimes\cR)^\times\simeq\cR^\times\times\cR^\times\times\cR^\times$ 
is by means of the $G$-grading on $\LL$ (the restriction of $\Gamma$). On $\mathbf{A}(\cR)$, this action coincides 
with the one given by the composition of $\eta'_\cR$ and the inner action of $\AAut_\FF(V,\LL,\beta,Q)(\cR)$ on $\mathbf{A}(\cR)$. 
Therefore, $\eta''=(\mathrm{d}a)^{-1}\eta'$. This proves that $\mathrm{H}^1(G^D,\mathbf{A})=1$.

\subsection{Lifting over real closed fields}

Let $\FF$ be a real closed field, so $\FF$ admits a unique ordering (in particular, $\chr{\FF}=0$), defined by 
$x\ge 0$ if and only if $x$ is a square, and the quadratic extension 
$\KK\bydef\FF[\sqrt{-1}]$ is the algebraic closure of $\FF$. 
Then $\LL$ is either $\FF\times\FF\times\FF$ or $\FF\times\KK$, but 
it is explained in the Introduction that there are no Type~III gradings in the first case
(since $\FF$ does not contain a primitive cube root of $1$). 
So, we assume $\LL=\FF\times\KK$. Let $\iota$ be the generator of $\cG=\Gal(\KK/\FF)$.

\begin{lemma}\label{lm:FxK}
For any $\lambda\in\LL^\times$, there exists $\ell\in\LL^\times$ such that $\ell^{-1}\ell^\sharp =\lambda$.
\end{lemma}

\begin{proof}
For any $\lambda=(b,c)\in\FF\times\KK$, we have $\lambda^\sharp =(c\iota(c),b\iota(c))$, 
and $c\iota(c)\ge 0$. It follows that we can find $\ell$ satisfying $\ell^2=\lambda^\sharp $. 
Then $N(\ell)^2=N(\ell^2)=N(\lambda^\sharp )=N(\lambda)^2$, so, 
replacing $\ell$ with $-\ell$ if necessary, we may assume $N(\ell)=N(\lambda)$. 
Then $\ell^{-1}\ell^\sharp =\ell^{-2}N(\ell)=(\lambda^\sharp )^{-1}N(\lambda)=\lambda$.
\end{proof}

Since $\LL$ is not a field, any twisted composition $(V,\LL,\beta,Q)$ of rank $8$ is similar to $\TC(\wb{\cC},\LL)$
where $(\wb{\cC},\bullet,n)$ is the para-Hurwitz algebra associated to a Cayley algebra $\cC$ over $\FF$. 
It follows from Lemma \ref{lm:FxK} that similar twisted compositions over $\LL$ are actually isomorphic,
so we may assume $(V,\LL,\beta,Q)=\TC(\wb{\cC},\LL)$.
Recall that $\TC(\wb{\cC},\LL)$ is obtained by Galois descent as in Equation \eqref{eq:twisted_comp} 
(with discriminant $\Delta\simeq\KK$) 
from the cyclic composition $(\wt{V},\wt{\LL},\rho,*,Q)$ where 
$\wt{V}=\wb{\cC}\otimes\wt{\LL}$ and $\wt{\LL}=\LL\otimes\KK\simeq\KK\times\KK\times\KK$.

Let $\mathbf{B}=\AAut_\FF(V,\LL,\beta,Q)$ and let $\mathbf{A}$ be the center of its connected component, 
$\Spins(V,\LL,\beta,Q)$. 
Note that $\mathbf{B}_\KK\simeq(\TRIs(\wb{\cC},\bullet,n)\rtimes\mathbf{S}_3)_\KK$.
We may assume that $\iota$ acts on $\wt{\LL}$ by sending 
$(\lambda_1,\lambda_2,\lambda_3)$ to $(\iota(\lambda_1),\iota(\lambda_3),\iota(\lambda_2))$. 
In particular, $\iota$ acts as the permutation $(2,3)$ on the center $A\bydef\mathbf{A}(\KK)$ of the group 
$\Spins(V,\LL,\beta,Q)(\KK)\simeq\Spins(\cC,n)(\KK)$. 

\begin{lemma}\label{lm:trivial_Galois_cohom}
$\mathrm{H}^1(\cG,A)=1$.
\end{lemma}

\begin{proof}
This can be checked directly from the definition or using the following well known fact: if $\cG$ is a cyclic group of order
$m$, $A$ is an abelian group, $\gamma$ is an automorphism of $A$ of period $m$, and $A$ is made a $\cG$-module by letting 
a generator of $\cG$ act by $\gamma$, then $\mathrm{H}^k(\cG,A)=\ker(1+\gamma+\cdots+\gamma^{m-1})/\im(1-\gamma)$ 
for all odd $k$. If $m=2$ and $A$ is an elementary $2$-group of rank $2$, then $1-\gamma=1+\gamma$
and $\ker(1+\gamma)=\im(1+\gamma)$ for any involutive operator $\gamma\ne 1$.
\end{proof}

Let $E=\End_\LL(V)$ be the trialitarian algebra corresponding to $(V,\LL,\beta,Q)$ and let $\mathbf{C}=\AAut_\FF(E,\LL,\sigma,\alpha)$.
The trialitarian algebra $\wt{E}$ corresponding to $(\wt{V},\wt{\LL},\rho,*,Q)$ 
can be identified with $E\otimes\KK$, and hence $\mathbf{C}_\KK=\AAut_\KK(\wt{E},\wt{\LL},\wt{\sigma},\wt{\alpha})$. 
Note that we have a short exact sequence $\mathbf{1}\to\mathbf{A}\to\mathbf{B}\to\mathbf{C}\to\mathbf{1}$ and,
for any affine scheme $\mathbf{H}$ over $\FF$, we have the corresponding exact sequence of $\cG$-groups
(i.e., groups equipped with a $\cG$-action by automorphisms):
$
1\to\mathrm{Aff}_\KK(\mathbf{H}_\KK,\mathbf{A}_\KK)
\to\mathrm{Aff}_\KK(\mathbf{H}_\KK,\mathbf{B}_\KK)
\to\mathrm{Aff}_\KK(\mathbf{H}_\KK,\mathbf{C}_\KK).
$ 

A Type~III grading $\Gamma:E=\bigoplus_{g\in G}E_g$ extends to a Type~III grading 
$\wt{\Gamma}:\wt{E}=\bigoplus_{g\in G}(E_g\otimes\KK)$. 
The homomorphism $\eta=\eta_\Gamma\colon G^D\to\mathbf{C}$ is identified with 
$\eta_{\wt{\Gamma}}$ if we regard $\mathrm{Aff}_\FF(G^D,\mathbf{C})$ as the fixed points of $\cG$
in $\mathrm{Aff}_\KK((G^D)_\KK,\mathbf{C}_\KK)$. 

By the results of the previous subsection, we know that the fiber of the map 
$\Hom_\KK((G^D)_\KK,\mathbf{B}_\KK)\to\Hom_\KK((G^D)_\KK,\mathbf{C}_\KK)$ over $\eta$ consists of four points, 
and if $\eta'$ is one of them, then the mapping $a\mapsto(\mathrm{d}a)\eta'$ is a bijection from $A$ onto the fiber. 
Note that the action of $G^D$ on $\mathbf{A}$ is $\cG$-equivariant and hence the injection 
$\mathrm{d}\colon A\to\mathrm{Aff}_\KK((G^D)_\KK,\mathbf{A}_\KK)$
is a homomorphism of $\cG$-modules. We want to prove that the fiber contains a fixed point of $\cG$,
which would be a desired lifting of $\eta$.
To this end, consider the $\cG$-group $M=\mathrm{Aff}_\KK((G^D)_\KK,\mathbf{B}_\KK)$, 
its normal $\cG$-subgroup $N=\mathrm{Aff}_\KK((G^D)_\KK,\mathbf{A}_\KK)$, 
and the quotient $M/N$, which is embedded in $\mathrm{Aff}_\KK((G^D)_\KK,\mathbf{C}_\KK)$.
The short exact sequence $1\to N\to M\to M/N\to 1$ gives rise to the exact sequence of pointed sets
\[
1\to N^\cG\to M^\cG\to (M/N)^\cG\stackrel{\delta^0}{\longrightarrow} \mathrm{H}^1(\cG,N)\to \mathrm{H}^1(\cG,M)\to \mathrm{H}^1(\cG,M/N),
\]
where the maps preceding $\delta^0$ are group homomorphisms (see e.g. \cite[\S 28.B]{KMRT}).
We have $\eta\in(M/N)^\cG$, so we have to show that $\delta^0(\eta)=1$. Since $\eta'\in M$ is a pre-image of $\eta$,
the definition of $\delta^0$ tells us that $\delta^0(\eta)$ is the class of the $1$-cocycle $\cG\to N$ 
given by $\sigma\mapsto(\eta')^{-1}(\sigma\cdot\eta')$. There exists a unique element $a_\sigma\in A$ such that 
$\sigma\cdot\eta'=(\mathrm{d}a_\sigma)\eta'$, hence $\delta^0(\eta)$ is the class of the $1$-cocycle $\sigma\mapsto\mathrm{d}a_\sigma$.
In other words, $\delta^0(\eta)=[\mathrm{d}_*(f)]$ where $f\colon\cG\to A$ is the map $\sigma\mapsto a_\sigma$ and $\mathrm{d}_*$ 
denotes the map $\mathrm{Fun}(\cG,A)\to\mathrm{Fun}(\cG,N)$ given by post-composition with the injection $\mathrm{d}\colon A\to N$. 
(This map is a part of the homomorphism of cochain complexes of $\cG$-modules induced by $\mathrm{d}$.)
Since $\mathrm{d}_*(f)$ is a $1$-cocycle, so is $f$. But $\mathrm{H}^1(\cG,A)=1$ by Lemma \ref{lm:trivial_Galois_cohom}, hence $f$ is a coboundary.
Therefore, $\mathrm{d}_*(f)$ is a coboundary, as required.

We have proved that $\eta$ has a lifting $\eta'$ over $\FF$. Then all such liftings have the form $(\mathrm{d}a)\eta'$ where $a$ is a fixed point in $A$,
or, in other words, $a\in\mathbf{A}(\FF)$. To summarize:

\begin{theorem}\label{th:real_tri_to_comp}
Let $\FF$ be a real closed field, $\KK=\FF[\sqrt{-1}]$, and $\LL=\FF\times\KK$.
Let $(V,\LL,\beta,Q)$ be a twisted composition of rank $8$ and let $E=\End_\LL(V)$ be the corresponding trialitarian 
algebra. Suppose that $E$ is given a Type III grading by an abelian group $G$. Then there exist exactly two 
$G$-gradings on the twisted composition $V$ that induce the given grading on $E$. If $V=\bigoplus_{g\in G}V_g$ is one 
of them, then the other is $V=\bigoplus_{g\in G}{aV_g}$ where $a=(1,-1)\in\LL$ is the unique nontrivial element of the center of the group $\Spins(V,\LL,\beta,Q)(\FF)$.\qed  
\end{theorem}

\subsection{Reduction of the classification of Type III gradings from trialitarian algebras to twisted composition algebras}

Since we want to classify gradings, there remains the question of the relationship between isomorphism 
of twisted compositions $(V,\LL,\beta,Q)$ and $(V',\LL',\beta',Q')$ equipped with Type~III gradings by $G$ 
and isomorphism of the corresponding $G$-graded trialitarian algebras $E=\End_\LL(V)$ and $E'=\End_{\LL'}(V')$.
We will treat this question for an arbitrary ground field $\FF$ (of characteristic different from $2,3$), 
but assuming that $\LL$ and $\LL'$ are not fields. Clearly, it is necessary for $\LL$ and $\LL'$ 
to be isomorphic as ungraded algebras, so we have two cases to consider: either $\LL$ and $\LL'$ are isomorphic to 
$\FF\times\FF\times\FF$ (type ${}^1D_4$) or $\LL$ and $\LL'$ are isomorphic to $\FF\times\KK$ where $\KK$ is a 
quadratic field extension of $\FF$ (type ${}^2D_4$).

An isomorphism $(\varphi_1,\varphi_0)\colon(V,\LL)\to(V',\LL')$ induces an algebra isomorphism 
$\varphi\colon\End_\LL(V)\to\End_{\LL'}(V')$ by the formula $\varphi(u)=\varphi_1 u\varphi_1^{-1}$ 
for all $u\in\End_\LL(V)$. If $(\varphi_1,\varphi_0)$ is a similitude $(V,\LL,\beta,Q)\to(V',\LL',\beta',Q')$
then $\varphi$ is an isomorphism of trialitarian algebras $(E,\LL,\rho,\sigma,\alpha)\to(E',\LL',\rho',\sigma',\alpha')$.
Note that $\rho$ and $\rho'$ are chosen arbitrarily and do not affect the set of isomorphisms.

In the case ${}^1D_4$, Type~III gradings do not exist unless $\FF$ contains a primitive cube root of unity.  
The next result is a restatement of \cite[Theorem 13]{EK15} in terms of twisted compositions 
rather than cyclic compositions.

\begin{theorem}\label{th:tri_isomorphism_1D4}
Let $(V,\LL,\beta,Q)$ and $(V',\LL',\beta',Q')$ be two twisted compositions of rank $8$ 
where $\LL=\FF\times\FF\times\FF$ and $\FF$ is a field of characteristic different from $2$ 
containing a primitive cube root of unity. 
Suppose $V$ and $V'$ are given Type~III gradings by an abelian group $G$. 
Let $E=\End_\LL(V)$ and $E'=\End_{\LL'}(V')$ be the corresponding trialitarian algebras with induced $G$-gradings. 
Then the mapping $(\varphi_1,\varphi_0)\mapsto\varphi$ is a bijection from the set of isomorphisms $V\to V'$ 
of graded twisted compositions onto the set of isomorphisms $E\to E'$ of graded trialitarian algebras.\qed
\end{theorem}

Combining this result with \cite[Theorem 12]{EK15} mentioned earlier, we obtain:

\begin{corollary}
Let $\LL=\FF\times\FF\times\FF$ where $\FF$ is an algebraically closed field, $\chr{\FF}\ne 2,3$, 
and fix a $3$-cycle $\rho$ permuting the components of $\LL$. Then the functor sending
$(V,\LL,\beta,Q)\mapsto\big(\End_\LL(V),\LL,\rho,\sigma,\alpha\big)$ and $(\varphi_1,\varphi_0)\mapsto\varphi$
is an equivalence from the groupoid of twisted compositions of rank $8$ equipped with a Type~III grading by an 
abelian group $G$ to the groupoid of trialitarian algebras equipped with a Type~III grading by $G$.\qed
\end{corollary}

We now turn to the case ${}^2D_4$. Since $\LL\otimes\KK\simeq\KK\times\KK\times\KK$, the existence of Type~III 
gradings requires that $\KK$ contain a primitive cube root of unity.

\begin{theorem}\label{th:tri_isomorphism_2D4}
Let $(V,\LL,\beta,Q)$ and $(V',\LL',\beta',Q')$ be two twisted compositions of rank $8$ 
where $\LL=\FF\times\KK$ and $\KK\supset\FF$ is a quadratic field extension of characteristic different from $2$ 
such that $\KK$ contains a primitive cube root of unity. 
Suppose $V$ and $V'$ are given Type~III gradings by an abelian group $G$. 
Let $E=\End_\LL(V)$ and $E'=\End_{\LL'}(V')$ be the corresponding trialitarian algebras with induced $G$-gradings. 
Then the mapping $(\varphi_1,\varphi_0)\mapsto\varphi$ is a bijection from the set of isomorphisms $V\to V'$ 
of graded twisted compositions onto the set of isomorphisms $E\to E'$ of graded trialitarian algebras.
\end{theorem}

\begin{proof}
It is clear that the mapping is well defined. To prove that a given $\varphi\colon E\to E'$
has a unique pre-image, we extend scalars from $\FF$ to $\KK$, apply Theorem \ref{th:tri_isomorphism_1D4} over $\KK$,
and observe that the uniqueness of $(\varphi_1,\varphi_0)$ ensures that it descends to $\FF$.
\end{proof}

\begin{corollary}
Let $\LL=\FF\times\KK$ where $\FF$ is a real closed field and $\KK=\FF[\sqrt{-1}]$, 
and fix a $3$-cycle $\rho$ permuting the components of $\LL\otimes\KK\simeq\KK\times\KK\times\KK$. 
Then the functor sending
$(V,\LL,\beta,Q)\mapsto\big(\End_\LL(V),\LL,\rho,\sigma,\alpha\big)$ and $(\varphi_1,\varphi_0)\mapsto\varphi$
is an equivalence from the groupoid of twisted compositions of rank $8$ equipped with a Type~III grading by an 
abelian group $G$ to the groupoid of trialitarian algebras equipped with a Type~III grading by $G$.
\end{corollary}

\begin{proof}
As explained in the Introduction, Type~III gradings can exist only on trialitarian algebras isomorphic to 
$\End_\LL(V)$ for some twisted composition $V$. It remains to apply Theorems \ref{th:real_tri_to_comp} and \ref{th:tri_isomorphism_2D4}.
\end{proof}

To finish this section, we consider fine gradings over a real closed field $\FF$. 
Fine gradings correspond to maximal diagonalizable subgroupschemes
in the appropriate automorphism group scheme, i.e., subgroupschemes that are maximal with respect to the 
property of being isomorphic to $G^D$ for some (finitely generated) abelian group $G$, which turns out to be 
the universal group of the corresponding grading (see \cite[Chapter 1]{EKmon}). 
Two fine gradings are equivalent if and only if the corresponding diagonalizable subgroupschemes belong 
to the same orbit of the inner action of the automorphism group (i.e., the group of $\FF$-points of the automorphism
group scheme).

Clearly, any refinement of a  Type~III grading is also a Type~III grading, hence Theorems \ref{th:real_tri_to_comp} 
and \ref{th:tri_isomorphism_2D4} imply that a fine Type~III grading on a twisted composition algebra $V$, 
realized over its universal group $G$, induces a fine Type~III grading on 
the trialitarian algebra $E=\End_\LL(V)$, for which $G$ is also the universal group. 
The following result is the analog for real closed fields of \cite[Theorem 15]{EK15}, with a completely analogous proof:

\begin{theorem}\label{th:real_tri_equivalence}
Let $\LL=\FF\times\KK$ where $\FF$ is a real closed field and $\KK=\FF[\sqrt{-1}]$, let $(V,\LL,\beta,Q)$ be 
a twisted composition of rank $8$, and let $E=\End_\LL(V)$ be the corresponding trialitarian algebra.
Then the mapping sending a fine grading on $V$, realized over its universal group, 
to the induced grading on $E$, gives a one-to-one correspondence 
between the equivalence classes of Type~III fine gradings on $V$ 
and the equivalence classes of Type~III fine gradings on $E$. This correspondence preserves universal groups.
\end{theorem}

\begin{proof}
Theorem \ref{th:tri_isomorphism_2D4} implies that two Type~III fine gradings on $V$ 
are equivalent if and only if the induced gradings on $E$ are equivalent.
Let $\Gamma$ be a Type~III fine grading on $E$ with universal group $G$. 
It is determined by a maximal diagonalizable subgroupscheme $\Qs$ of $\AAut_\FF(E,\LL,\rho,\sigma,\alpha)$, 
which we may identify with $G^D$. By Theorem \ref{th:real_tri_to_comp}, we can induce $\Gamma$ from a Type~III 
$G$-grading $\Gamma'$ on $V$. We have to prove that $\Gamma'$ is fine. 
Let $\Qs'$ be the corresponding diagonalizable subgroupscheme of $\AAut_\FF(V,\LL,\beta,Q)$. 
Then the homomorphism $\mathrm{Int}$ restricts to an isomorphism $\Qs'\to\Qs$ (see Diagram \eqref{eq:eta_diag}).  
Assume that $\Qs'$ is not maximal. Then there exists diagonalizable $\tilde{\Qs}'$ properly containing $\Qs'$. 
The image $\mathrm{Int}(\tilde{\Qs}')$ is necessarily $\Qs$ because the latter is maximal. 
We will obtain a contradiction if we can show that the intersection of $\tilde{\Qs}'$ with the kernel $\mathbf{A}$ 
of $\mathrm{Int}$ is trivial. It suffices to show that the intersection has no $\KK$-points different from the identity. 
But this is clear since $\mathbf{A}(\KK)$ is an elementary $2$-group of rank $2$ and $\Qs'(\KK)$ contains 
an automorphism whose inner action cyclically permutes the non-identity elements of $\mathbf{A}(\KK)$.
\end{proof}

\section{Type~III gradings on real twisted compositions}\label{s:Type_III}

\subsection{Cayley gradings on twisted Hurwitz compositions of rank $8$}

Let $\cC$ be a Cayley algebra over $\RR$ (or any real closed field), with norm $n$, multiplication $xy$ and standard involution $x\mapsto \bar x=n(1,x)1-x$. Recall that the associated para-Cayley algebra $\overline{\cC}$ coincides with $\cC$ as a quadratic space, but its multiplication is given by $x\bullet y=\bar x\bar y$ for any $x,y\in \cC$. 
Consider $\LL=\RR\times\CC$, its discriminant algebra $\Delta\simeq \CC$, and its $S_3$-Galois closure $\Sigma\simeq\LL\otimes\CC$. Recall from Subsection \ref{ss:twisted_compositions} that, for a fixed $3$-cycle $\rho\in\Gal(\Sigma/\RR)$, we have a cyclic composition $\bigl(\cC\otimes\LL\otimes\CC,\LL\otimes\CC,\rho,*,Q\bigl)$ with ground field $\CC$. The corresponding twisted composition $\bigl(\cC\otimes\LL\otimes\CC,\LL\otimes\CC,\beta,Q\bigl)$ does not depend on the choice of $\rho$. We record for future reference: 
\begin{equation}\label{eq:TC}
\begin{split}
&Q(x\otimes\ell\otimes c)=n(x)\ell^2\otimes c^2,\\
&b_Q(x\otimes\ell\otimes c,x'\otimes\ell'\otimes c')=n(x,x')\ell\ell'\otimes cc',\\
&(x\otimes\ell\otimes c)*(x'\otimes\ell'\otimes c')=(x\bullet x')\otimes \rho(\ell\otimes c)\rho^2(\ell'\otimes c'),\\
&\beta(X)=X*X,
\end{split}
\end{equation}
for all $x,x'\in \cC$, $\ell,\ell'\in \LL$, $c,c'\in \CC$ and $X\in \cC\otimes\LL\otimes\CC$.

Consider the twisted composition $(V,\LL,\beta,Q)=\TC(\overline{\cC},\LL)$ with ground field $\RR$, which is obtained by Galois descent as in Equation \eqref{eq:twisted_comp}. Here $\beta$ and $Q$ denote the restrictions to $V$ of the corresponding maps on $\cC\otimes\LL\otimes\CC$. 
The descent is given by the involutive $\CC$-antilinear operator $\tilde\iota=\bar{}\otimes\id\otimes\iota$, where $\iota$ is the generator of $\Gal(\CC/\RR)$,
so $\iota(\bi)=-\bi$ for the imaginary unit $\bi\in\CC$. Thus,
\begin{equation}\label{eq:VTC}
V=\{x\in \cC\otimes\LL\otimes\CC \mid \tilde\iota(x)=x\}
 =(\RR 1\otimes\LL\otimes 1)\oplus(\cC^0\otimes\LL\otimes\bi),
\end{equation}
where $\cC^0:=\{x\in \cC\mid n(x,1)=0\}=\{x\in \cC\mid \bar x=-x\}$.

Fix a primitive cube root of unity $\omega\in\CC$ and let $\xi=(1,\omega)\in\LL$.
Given an abelian group $G$, an element $h\in G$ of order $3$ and a $G$-grading $\Gamma_\cC$ on $\cC$, $h$ determines a $G$-grading on $\LL$ by setting 
$\deg(\xi):=h$, and hence $\Gamma_\cC$ and $h$ define a $G$-grading of Type~III on $(\cC\otimes\LL\otimes\CC,\LL\otimes\CC,\beta,Q)$ whose homogeneous components are given by
\[
(\cC\otimes\LL\otimes\CC)_g\bydef \bigl(\cC_g\otimes \LL_e\otimes\CC\bigr)\,\oplus\,\bigl(\cC_{gh^2}\otimes\LL_h\otimes\CC\bigr)\,\oplus\, \bigl(\cC_{gh}\otimes\LL_{h^2}\otimes\CC\bigr).
\]
Since the components are invariant under $\tilde\iota$, this grading restricts to a $G$-grading of Type~III on $\TC(\overline{\cC},\LL)$, which we denote by $\Gamma\bigl(G,\Gamma_\cC,h\bigr)$. We will refer to gradings of this form as \emph{Cayley gradings}.

\begin{lemma}\label{le:h_h2}
$\Gamma\bigl(G,\Gamma_\cC,h\bigr)$ is isomorphic to $\Gamma\bigl(G,\Gamma_\cC,h^2\bigr)$.
\end{lemma}
\begin{proof} Let $\tau$ be the nontrivial automorphism of $\LL$ over $\RR$, so $\tau(\xi)=\xi^2$. Then, since $(\tau\otimes\id)\circ\rho=\rho^2\circ(\tau\otimes\id)$ on $\LL\otimes\CC$, the $\CC$-linear operator $\;\bar{}\otimes\tau\otimes\id$ is an automorphism of $(\cC\otimes\LL\otimes\CC,\LL\otimes\CC,\beta,Q)$. 
Since it commutes with $\tilde\iota$, it restricts to an automorphism of $\TC(\overline{\cC},\LL)$. But the standard involution preserves the components of $\Gamma_\cC$ while $\tau(\xi)=\xi^2$, so this automorphism sends $\Gamma\bigl(G,\Gamma_\cC,h\bigr)$ to $\Gamma\bigl(G,\Gamma_\cC,h^2\bigr)$.
\end{proof}

We are going to prove that any Type~III grading on a real twisted composition of rank $8$ is isomorphic to a Cayley grading. Moreover, we will see that, in most cases, the graded Cayley algebra $\cC$ can be obtained in a canonical way.

\subsection{Reduction to Cayley gradings}

Let $(V,\LL,\beta,Q)$ be a real twisted composition of rank $8$ with $\LL=\RR\times\CC$. As before, fix a primitive cube root of unity $\omega\in\CC$ and let $\xi=(1,\omega)\in\LL$. Let $G$ be an abelian group.

\begin{lemma}\label{le:e}
Let $\Gamma$ be a $G$-grading of Type~III on $(V,\LL,\beta,Q)$. Then there exists an element $\veps\in V_e$ such that $\beta(\veps)=\veps$ and $Q(\veps)=1$.
\end{lemma}

\begin{proof}
By \cite[Theorem 38.6]{KMRT}, the vector space $J=\LL\oplus V$ is an exceptional cubic Jordan algebra (or Albert algebra), whose norm, adjoint and trace form are given as follows (recall that $N_V(v)\bydef b_Q\bigl(v,\beta(v)\bigr)$ belongs to the ground field):
\begin{equation}\label{eq:norm_trace}
\begin{split}
N\bigl((\ell,v)\bigr)&=N_{\LL/\RR}(\ell)+b_Q\bigl(v,\beta(v)\bigr)-T_{\LL/\RR}\bigl(\ell Q(v)\bigr),\\
(\ell,v)^\sharp&=\bigl(\ell^\sharp-Q(v),\beta(v)-\ell v\bigr),\\
T\bigl((\ell,v),(\ell',v')\bigr)&=T_{\LL/\RR}\bigl(\ell\ell'\bigr)+T_{\LL/\RR}\bigl(b_Q(v,v')\bigr).
\end{split}
\end{equation}
In particular, $V$ is the orthogonal complement to $\LL$ relative to the trace form.

Our grading $\Gamma$ induces a grading on the Albert algebra $J$, and $J_e=\LL_e\oplus V_e=\RR 1\oplus V_e$. Let us first check that $V_e\neq 0$. Otherwise $\dim J_e=1$ and hence, applying \cite[Theorem 5.12]{EKmon} (see also \cite{EK12a}) to $J\otimes\CC$, we get that the support of the grading is an elementary abelian subgroup isomorphic to $\ZZ_3^3$ and $\dim J_g=1$ for any $g$ in the support. Now the argument in the proof of \cite[Lemma 5.21]{EKmon} shows that there are generators $h_1,h_2,h_3$ of the support such that, for any nonzero elements $X_i\in J_{h_i}\otimes\CC$, $i=1,2,3$, we have $0\neq (X_1X_2)X_3=\omega X_1(X_2X_3)$. But we may take $X_i\in J_{h_i}\simeq J_{h_i}\otimes 1$, so that both $(X_1X_2)X_3$ and $X_1(X_2X_3)$ are in the one-dimensional space $J_{h_1h_2h_3}$, and we get a contradiction because $\omega\not\in\RR$.

Now \cite[Corollary 5.2]{EKmon} shows that $J_e\otimes\CC$ is a semisimple cubic Jordan algebra of dimension at least $2$, hence so is $J_e=\RR 1\oplus V_e$,
and $V_e=\{x\in J_e\mid T(x)=0\}$ where $T(x):=T(x,1)$. Hence, there is a Zariski-open nonempty subset $U$ of $J_e$ such that the (associative) subalgebra $\RR[x]$ generated by any $x\in U$ is \'etale \cite[Lemma 38.2(2)]{KMRT}. For $x\in U\setminus\RR 1$, $\RR[x]$ is then isomorphic to either $\RR\times\RR$ with $N\bigl((a,b)\bigr)=ab^2$, $\RR\times\RR\times\RR$ with $N\bigl((a,b,c)\bigr)=abc$, or $\RR\times\CC$ with $N\bigl((b,c)\bigr)=bc\bar{c}$. The element $\veps\in \RR[x]$ corresponding to
\[
\begin{cases}
(2,-1)&\text{if $\RR[x]\simeq\RR\times\RR$ or $\RR\times\CC$,}\\
(2,-1,-1)&\text{if $\RR[x]\simeq\RR\times\RR\times\RR$,}
\end{cases}
\]
is in $V_e$, and $\veps^\sharp$ corresponds to either $(1,-2)$ or $(1,-2,-2)$, so  $\veps^\sharp =-1+\veps$ and hence, by \eqref{eq:norm_trace}, $Q(\veps)=1$ and $\beta(\veps)=\veps$.

It must be remarked that $\RR\times\RR$ is a quadratic \'etale algebra, but it is treated here as cubic, being a subalgebra of the Jordan algebra $J$ of degree $3$, with the cubic norm obtained as the restriction of the norm $N$ of $J$. The trace is given by $T\bigl((a,b)\bigr)=a+2b$ and the adjoint by $(a,b)^\sharp=(b^2,ab)$. These are the restrictions to $\RR\times\RR$ of the corresponding maps on $\RR\times\CC$ (or on $\RR\times\RR\times\RR$, if we embed $\RR\times\RR$ by $(a,b)\mapsto (a,b,b)$).
\end{proof}

Consider the twisted composition $(V\otimes\CC,\LL\otimes\CC,\beta,Q)$ with ground field $\CC$, where $\beta$ and $Q$ are extended from $V$.
As before, let $\iota$ be the generator of $\Gal(\CC/\RR)$, i.e., the complex conjugation. Let $\rho$ be the generator of $\Gal(\LL\otimes\CC/\CC)\simeq A_3$ that satisfies $\rho(\xi\otimes 1)=\xi\otimes\omega$ and let $\beta_\rho(x,y):=x*y$ be  the unique structure of cyclic composition on $V\otimes\CC$ with respect to $\rho$ such that $\beta(x)=x*x$ for all $x\in V\otimes\CC$.

Let $\Gamma$ be a $G$-grading of Type~III on $(V,\LL,\beta,Q)$. Take $\veps\in V_e$ as in Lemma~\ref{le:e} and denote by $\pi$ the negative of the corresponding  hyperplane reflection in $V$, i.e., $\pi(x)=b_Q(x,\veps)\veps -x$. Also define the $\CC$-linear operator $\varphi$ on $V\otimes\CC$ by $\varphi(x)= \beta_\rho\bigl((\pi\otimes\id)(x),\veps\otimes 1\bigr)$.

As shown in the proof of $(3)\Rightarrow (1)$ in \cite[Theorem 36.24]{KMRT}, we have
\begin{equation*} 
\begin{gathered}
\varphi^2(x)=\beta_\rho\bigl(\veps\otimes 1,(\pi\otimes\id)(x)\bigr),\ \varphi^3=\id,\\ 
\varphi(\pi\otimes\id)=(\pi\otimes\id)\varphi,\ (\id\otimes\iota)\varphi(\id\otimes\iota)=\varphi^2,\\
(\id\otimes\iota)\beta_\rho(x,y)=\beta_\rho\bigl((\pi\otimes\id)(y),(\pi\otimes\id)(x)\bigr).
\end{gathered}
\end{equation*}
It follows that the operators $\varphi$ and $\pi\otimes\iota$ define a semilinear action of $\Gal(\LL\otimes\CC/\RR)\simeq S_3$ on $V\otimes\CC$. Consider the corresponding $\RR$-form of the free $(\LL\otimes\CC)$-module $V\otimes\CC$:
\begin{equation}\label{eq:C}
\cC\bydef\{x\in V\otimes\CC \mid \varphi(x)=x=(\pi\otimes\iota)(x)\},
\end{equation}
so we have an $(\LL\otimes\CC)$-linear isomorphism
\begin{equation}\label{eq:Phi}
\begin{split}
\Phi\colon\cC\otimes(\LL\otimes\CC)&\longrightarrow V\otimes\CC,\\
  x\otimes \alpha\ &\mapsto\quad x\alpha.
\end{split}
\end{equation}
Moreover, $\cC$ is a Cayley algebra over $\RR$ with unity $\veps\otimes 1$ (which, by abuse of notation, will be denoted by $\veps$), multiplication given by \cite[p.~503]{KMRT}:
\[
xy=\beta_\rho\bigl((\pi\otimes\id)(x),(\pi\otimes\id)(y)\bigr),
\]
and norm $n$ given by the restriction to $\cC$ of the extension of $Q$ to $V\otimes\CC$. We record for future reference:

\begin{remark}\label{rem:beta_on_C}
The norm and the standard involution on $\cC$ are given by $n(x)=Q(x)$ and $\bar x=(\pi\otimes\id)(x)$.
Also, for any $x,y\in \cC$, we have $\beta_\rho(x,y)=\bar x\bar y=x\bullet y$ (the para-Cayley multiplication).
\end{remark}

Note that the operator $\id\otimes\iota$ on $V\otimes\CC$ corresponds under $\Phi$ to the operator $\tilde\iota=\bar{}\otimes\id\otimes\iota$ on $\cC\otimes\LL\otimes\CC$ and hence $V\simeq V\otimes 1$ corresponds under $\Phi$ to $\tilde V=(\RR 1\otimes\LL\otimes 1)\oplus(\cC^0\otimes\LL\otimes\bi)$, where $\cC^0=\{x\in \cC\mid n(x,\veps)=0\}$. Therefore, $\Phi$ restricts to the following isomorphism, which we denote by the same letter:
\[
\Phi\colon\TC(\overline{\cC},\LL)\rightarrow (V,\LL,\beta,Q).
\]
The given grading $\Gamma$ on $(V,\LL,\beta,Q)$ extends to a $G$-grading on $(V\otimes\CC,\LL\otimes\CC,\beta,Q)$. Being of Type~III, the latter grading corresponds to a homomorphism $\eta\colon G^D\rightarrow \AAut_\CC(V\otimes\CC,\LL\otimes\CC,\beta,Q)$ whose image is contained in the inverse image of $\mathbf{A}_3$ under the projection $\AAut_\CC(V\otimes\CC,\LL\otimes\CC,\beta,Q)\rightarrow \AAut_\CC(\LL\otimes\CC)\simeq \mathbf{S}_3$, which is $\AAut_\CC(V\otimes\CC,\LL\otimes\CC,\rho,\beta_\rho,Q)$. Hence, not only $b_Q$ but also $\beta_\rho$ is a homogeneous map of degree $e$. Since $\veps\in V_e$, it follows that the maps $\pi$ and $\varphi$ are homogeneous of degree $e$, hence the Cayley algebra $\cC$ inherits a $G$-grading from $V\otimes\CC$. Moreover, $\Phi$ becomes a graded isomorphism if we define a $G$-grading on $\cC\otimes\LL\otimes\CC$ by combining the gradings on $\cC$ and on $\LL$ induced by $\Gamma$.

In summary:

\begin{theorem}\label{th:typeIII_C}
Let $\Gamma$ be a $G$-grading of Type~III on a real twisted composition $(V,\LL,Q,\beta)$ of rank $8$ with $\LL=\RR\times\CC$. Then there exists a Cayley algebra $\cC$ over $\RR$, an element $h\in G$ of order $3$ and a $G$-grading $\Gamma_\cC$ on $\cC$ such that $\Gamma$ is isomorphic to the grading $\Gamma(G,\Gamma_\cC,h)$ on $\TC(\overline{\cC},\LL)$. \qed
\end{theorem}

\subsection{Classification of Type~III gradings}

With notation as in the previous subsection, observe that $\beta_\rho$ restricts to a $\CC$-bilinear product $*$ on $V_e\otimes\CC$ and the restriction of $Q$ to $V_e\otimes\CC$ takes values in $\LL_e\otimes\CC=\CC$. Hence, Equation \eqref{eq:cyclic_comp_id} for $V_e\otimes\CC$ becomes Equation \eqref{eq:symmetric_comp_id} with $*$ as the product and $Q$ as the norm, so $(V_e\otimes\CC,*,Q)$ is a symmetric composition algebra over $\CC$, and $\veps\simeq \veps\otimes 1$ is an idempotent of norm $1$ in this algebra. In particular, $\dim_\RR V_e=1,2,4$ or $8$, and we are going to consider each of these possibilities. Denote, as before, $\xi=(1,\omega)$ and let $h=\deg(\xi)$ with respect to the $G$-grading induced by $\Gamma$ on $\LL$ (so $h\in G$ has order $3$). The following observation will be useful:

\begin{lemma}\label{lm:para-unit}
Let $\cC$ be the $\RR$-form of $V\otimes\CC$ defined by Equation~\eqref{eq:C}. If $\veps\otimes 1$ is a para-unit of $V_e\otimes\CC$, then 
$\cC_e=(\RR\veps)\otimes 1\oplus(\RR\veps)^\perp\otimes\bi$, where the orthogonal complement is taken in $V_e$ with respect to $Q$.
In particular, $\dim_\RR \cC_e=\dim_\RR V_e$.
\end{lemma}

\begin{proof}
By definition of para-units, for any $x\in V_e\otimes\CC$, we have $x*(\veps\otimes 1)=(\pi\otimes\id)(x)$ and hence
\[
\varphi(x)=\beta_\rho\bigl((\pi\otimes\id)(x),\veps\otimes 1\bigr)=(\pi\otimes\id)^2(x)=x.
\]
Since $\varphi$ is the action of a generator of $\Gal(\LL\otimes\CC/\CC)$ on $V\otimes\CC$, we obtain 
$V_e\otimes\CC=\CC\cC_e$, which corresponds to $\cC_e\otimes \RR 1\otimes\CC$ under the isomorphism $\Phi$ of Equation~\eqref{eq:Phi}.
It remains to recall that the generator of $\Gal(\CC/\RR)$ acts on $V\otimes\CC$ as $\pi\otimes\iota$.
\end{proof}

\bigskip

\noindent $\boxed{\dim V_e=1}$ In this case, $V_e=\RR\veps=\cC_e$, so the restriction $\Gamma_\cC$ of the (complexified) grading $\Gamma$ to $\cC$ is isomorphic to $\Gamma_\OO(G,\RR,T)$ or $\Gamma_{\OO_s}(G,\RR,T,\mu)$, where $T\simeq\ZZ_2^3$ and $\mu\colon T\rightarrow \{\pm 1\}$ is a nontrivial homomorphism
(see Corollary \ref{co:real_Cayley} in Section \ref{s:octonions_G2}).

\bigskip

\noindent $\boxed{\dim V_e=2}$ Then the product $*$ is commutative and, since $\beta(x)=x*x$ for any $x\in V\simeq V\otimes 1$, we have $x*y=\frac{1}{2}\Bigl(\beta(x+y)-\beta(x)-\beta(y)\Bigr)\in V_e$ for any $x,y\in V_e$. Hence, $(V_e,*,Q)$ is a two-dimensional symmetric composition algebra over $\RR$, so it is a para-Hurwitz algebra (see e.g. \cite[Proposition 4.43]{EKmon}). Up to isomorphism, 
there are two possibilities: the para-Hurwitz algebras $\overline{\RR\times\RR}$ and $\overline{\CC}$. 
In the first case $\veps$ is the unique idempotent of $(V_e,*,Q)$, which corresponds to the para-unit $(1,1)\in\overline{\RR\times\RR}$, while in the second case $(V_e,*,Q)$ has three different nonzero idempotents, which correspond to $1,\omega,\omega^2\in\overline{\CC}$ and are all para-units.

\noindent $\bullet$ If $V_e$ is isomorphic to $\overline{\RR\times\RR}$, then $V_e=\RR\veps\oplus\RR x$, with $b_Q(x,\veps)=0$ and $x*x=\veps$ (as $x$ corresponds to $(1,-1)\in\overline{\RR\times\RR}$). Now Lemma~\ref{lm:para-unit} gives $\cC_e=\RR(\veps\otimes 1)\oplus\RR(x\otimes\bi)$, and $(x\otimes\bi)^2=-(\veps\otimes 1)$, so $\cC_e$ is isomorphic to $\CC$. Hence, the restriction $\Gamma_\cC$ is isomorphic to $\Gamma_\OO(G,\CC,T)$ or $\Gamma_{\OO_s}(G,\CC,T,\mu)$, where $T\simeq\ZZ_2^2$ and $\mu\colon T\rightarrow \{\pm 1\}$ is a nontrivial homomorphism.

\noindent $\bullet$ If $V_e$ is isomorphic to $\overline{\CC}$, let $\veps_1,\veps_2,\veps_3$ be the nonzero idempotents. They are permuted by $\Aut_\RR(V_e,*,Q)$ and satisfy $\veps_1+\veps_2+\veps_3=0$ and $Q(\veps_i)=1$ for $i=1,2,3$. Suppose we have chosen $\veps=\veps_1$.
Since $b_Q(\veps_1,\veps_2-\veps_3)=0$ and $V_e=\RR\veps_1\oplus\RR(\veps_2-\veps_3)$, Lemma~\ref{lm:para-unit} gives 
$\cC_e=\RR(\veps_1\otimes 1)\oplus\RR((\veps_2-\veps_3)\otimes\bi)$, with unity $\veps_1\otimes 1$. 
There is an isomorphism $V_e\to\overline{\CC}$ sending $\veps_1\mapsto 1$, $\veps_2\mapsto\omega$, $\veps_3\mapsto\omega^2$, and $(\omega-\omega^2)^{\bullet 2}=(\omega^2-\omega)^2=\omega^2+\omega-2=-3$, so $\bigl((\veps_2-\veps_3)\otimes\bi\bigr)^{*2}=3(\veps_1\otimes 1)$.

Recall that $Q$ restricts to the norm of $\cC$ and $\beta_\rho$ to the para-Hurwitz product on $\cC$ (Remark \ref{rem:beta_on_C}). Hence, 
$E_1\bydef\frac12\bigl(\veps_1\otimes 1+\frac{1}{\sqrt{3}}(\veps_2-\veps_3)\otimes\bi\bigr)$ and  
$E_2\bydef\frac12\bigl(\veps_1\otimes 1-\frac{1}{\sqrt{3}}(\veps_2-\veps_3)\otimes\bi\bigr)$ 
are orthogonal idempotents of $\cC$, so $\cC$ must be the split Cayley algebra $\OO_s$. To be specific, suppose $\omega=-\frac12 +\frac{\sqrt{3}}{2}\bi$.
Then one checks easily that
\[
E_1=\frac13 \bigl(\veps_1\otimes 1+\veps_2\otimes\omega+\veps_3\otimes \omega^2\bigr),\;
E_2=\frac13 \bigl(\veps_1\otimes 1+\veps_2\otimes\omega^2+\veps_3\otimes \omega\bigr),
\] 
and also $\veps_1\otimes 1=E_1+E_2$, $\veps_2\otimes 1=\omega^2E_1+\omega E_2$, and $\veps_3\otimes 1=\omega E_1+\omega^2E_2$.

Therefore, the restriction $\Gamma_\cC$ is isomorphic to $\Gamma_{\OO_s}(G,\gamma)$ with $\gamma=(g_1,g_2,g_3)\in G^3$ such that $g_1g_2g_3=e$, 
so there is a good basis $\{ E_1,E_2,u_1,u_2,u_3,w_1,w_2,w_3 \}$ of $\cC$ with $\deg(u_i)=g_i=\deg(v_i)^{-1}$ for $i=1,2,3$ (see Section \ref{s:octonions_G2}). 
Thus, 
\[
\cC=\RR E_1\oplus\RR E_2\oplus U\oplus W
\] 
where $U=\lspan{u_1,u_2,u_3}$ and $W=\lspan{w_1,w_2,w_3}$, which are the Peirce components of $\cC$ relative to $E_1$ and $E_2=1-E_1$. We have $g_i\neq e$ for all $i=1,2,3$ because $\cC_e=\RR E_1\oplus\RR E_2$. Moreover, $g_i\not\in\langle h\rangle$ because $\dim_\CC (V_e\otimes\CC)=2$ and $V_e\otimes\CC$ corresponds, under the isomorphism $\Phi$ of Equation~\eqref{eq:Phi}, to 
\[
\begin{split}
&(\cC_e\otimes\LL_e\otimes \CC)\oplus(\cC_h\otimes\LL_{h^2}\otimes\CC)\oplus(\cC_{h^2}\otimes\LL_h\otimes\CC)\\
&=(\cC_e\otimes\RR 1\otimes \CC)\oplus(\cC_h\otimes\RR\xi^2\otimes\CC)\oplus(\cC_{h^2}\otimes\RR\xi\otimes\CC),
\end{split}
\] 
which implies $\cC_h=\cC_{h^2}=0$. Thus, $\gamma\in (G\setminus\langle h\rangle)^3$.

Note that
\[
\begin{split}
U&=\{x\in \cC\mid b_Q(\veps_i\otimes 1,x)=0,\ i=1,2,3,\ E_1x=x,\ E_2x=0\}\\
 &=\{x\in \cC\mid b_Q(\veps_i\otimes 1,x)=0,\ i=1,2,3,\ E_2\bullet x=-x,\ E_1\bullet x=0\}.
\end{split}
\]
Hence, for $x\in U$, we get
\[
\begin{split}
\beta_\rho(\veps_1\otimes 1,x)&=(\veps_1\otimes 1)\bullet x=(E_1+E_2)\bullet x=-x,\\
\beta_\rho(\veps_2\otimes 1,x)&=\omega^2\beta_\rho(E_1,x)+\omega\beta_\rho(E_2,x)=\omega^2 E_1\bullet x+\omega E_2\bullet x=-\omega x,\\
\beta_\rho(\veps_3\otimes 1,x)&=\omega\beta_\rho(E_1,x)+\omega^2\beta_\rho(E_2,x)=\omega E_1\bullet x+\omega^2 E_2\bullet x=-\omega^2 x.
\end{split}
\]
Conversely, if $x\in V\otimes\CC$ satisfies $\beta_\rho(\veps_i\otimes 1,x)=-\omega^{i-1}x$ and $b_Q(\veps_i\otimes 1,x)=0$, for $i=1,2,3$, and $(\pi\otimes\iota)(x)=x$, then $\varphi^2(x)=\beta_\rho\bigl(\veps_1\otimes 1,(\pi\otimes\id)(x)\bigr)=-\beta_\rho(\veps_1\otimes 1,x)=x$, 
hence $\varphi(x)=x$ and $x\in \cC$. Besides, 
\begin{align*}
E_1\bullet x&=\beta_\rho(E_1,x)=\frac13\beta_\rho\bigl(\veps_1\otimes 1+\veps_2\otimes\omega+\veps_3\otimes \omega^2,x\bigr)
=-\frac13(1+\omega+\omega^2)x=0,\\
E_2\bullet x&=\beta_\rho(E_2,x)=\frac13\beta_\rho\bigl(\veps_1\otimes 1+\veps_2\otimes\omega^2+\veps_3\otimes \omega,x\bigr)
=-\frac13(1+1+1)x=-x.
\end{align*}
We have proved:
\[
\begin{split}
U=\{x\in V\otimes\CC\mid\, &(\pi\otimes\iota)(x)=x,\\ 
&b_Q(\veps_i\otimes 1,x)=0, \beta_\rho(\veps_i\otimes 1,x)=-\omega^{i-1}x,\, i=1,2,3\}.
\end{split}
\]
In the same vein, we obtain:
\[
\begin{split}
W=\{x\in V\otimes\CC\mid\, & (\pi\otimes\iota)(x)=x,\\ 
&b_Q(\veps_i\otimes 1,x)=0,\beta_\rho(x,\veps_i\otimes 1)=-\omega^{i-1}x,\, i=1,2,3\}.
\end{split}
\]
If, instead of $\veps=\veps_1$, we choose $\veps=\veps_2$, then the Cayley algebra $\cC$ will be replaced by $\tilde{\cC}$, 
with $\tilde{\cC}_e=\RR\tilde{E}_1\oplus\RR\tilde{E}_2$, where the orthogonal idempotents are now
\[
\tilde E_1=\frac13\bigl(\veps_2\otimes 1+\veps_3\otimes\omega+\veps_1\otimes\omega^2)=\omega^2E_1,\;
\tilde E_2=\frac13\bigl(\veps_2\otimes 1+\veps_3\otimes\omega^2+\veps_1\otimes\omega)=\omega E_2,
\]
and the corresponding Peirce components are
\[
\begin{split}
\tilde{U}=\{x\in V\otimes\CC\mid\, & (\pi\otimes\iota)(x)=x,\\ 
&b_Q(\veps_i\otimes 1,x)=0,\beta_\rho(\veps_i\otimes 1,x)=-\omega^{i-2}x,\, i=1,2,3\},\\
\tilde{W}=\{x\in V\otimes\CC\mid\, & (\pi\otimes\iota)(x)=x,\\ 
&b_Q(\veps_i\otimes 1,x)=0,\beta_\rho(x,\veps_i\otimes 1)=-\omega^{i-2}x,\, i=1,2,3\}.
\end{split}
\]
But recall the element $\xi=(1,\omega)\in \LL$ of degree $h$. For any $x\in U$, we have:
\[
\beta_\rho\bigl(\veps_i\otimes 1,x(\xi\otimes 1)\bigr)=\beta_\rho(\veps_i\otimes 1,x)\rho^2(\xi\otimes 1)
=(-\omega^{i-1}x)(\xi\otimes\omega^2)=-\omega^{i-2}\bigl(x(\xi\otimes 1)\bigr),
\]
so we get $\tilde{U}=U(\xi\otimes 1)$ and, similarly, $\tilde{W}=W(\xi^2\otimes 1)$. Thus, 
\[
\tilde \cC=\RR\tilde E_1\oplus\RR\tilde E_2\oplus U(\xi\otimes 1)\oplus W(\xi^2\otimes 1).
\]
The restriction $\Gamma_{\tilde{\cC}}$ is then isomorphic to $\Gamma_{\OO_s}(G,\tilde{\gamma})$ with $\tilde{\gamma}=(g_1h,g_2h,g_3h)$.

\bigskip

\noindent $\boxed{\dim V_e=4}$ Then $(V_e\otimes\CC,*,Q)$ is a four-dimensional symmetric composition algebra over $\CC$, hence it is a para-quaternion algebra
(see e.g. \cite[Theorem 4.44]{EKmon}), which contains a unique para-unit. 
Since $\id\otimes\iota$ is a $\CC$-antilinear automorphism, it must fix this unique para-unit, which therefore 
must belong to $V_e\simeq V_e\otimes 1$. This gives a canonical choice of the element $\veps$ in Lemma \ref{le:e}, namely, the para-unit of $(V_e\otimes\CC,*,Q)$. 
Then, by Lemma~\ref{lm:para-unit}, we have $\dim\cC_e=4$. It follows as in the previous case that $\cC_h=\cC_{h^2}=0$. 
If the quaternion algebra $\cC_e$ is split (i.e., isomorphic to $M_2(\RR)$), then $\Gamma_\cC$ is isomorphic to $\Gamma_{\OO_s}(G,\gamma)$ with $\gamma=(e,g,g^{-1})$, $g\in G\setminus\langle h\rangle$. Otherwise $\Gamma_\cC$ is isomorphic to either $\Gamma_\OO(G,\HH,T)$ or $\Gamma_{\OO_s}(G,\HH,T)$, where $T\simeq\ZZ_2$.

\bigskip

\noindent $\boxed{\dim V_e=8}$ Then $(V_e\otimes\CC,*,Q)$ is either a para-Cayley or Okubo algebra over $\CC$
(see e.g. \cite[Theorem 4.44]{EKmon}).

\noindent $\bullet$ If $(V_e\otimes\CC,*,Q)$ is a para-Cayley algebra then, as above, we may take $\veps$ to be the para-unit of this algebra, and then $\cC=\cC_e$, i.e., $\Gamma_\cC$ is trivial.

\noindent $\bullet$ Otherwise, $\veps\otimes 1$ is an idempotent of the Okubo algebra $(V_e\otimes\CC,*,Q)$, and $\varphi\colon x\mapsto\beta_\rho\bigl((\pi\otimes\id)(x),\veps\otimes 1\bigr)$ is an order $3$ automorphism of $(V_e\otimes\CC,*,Q)$. By \cite[Proposition 3.4 and Theorem 3.5]{EPI96} (see also \cite{Eld17}), the subalgebra $\{x\in V_e\otimes\CC\mid \varphi(x)=x\}$ has dimension $4$ over $\CC$ (a para-quaternion algebra), hence $\cC_e=\{x\in V_e\otimes\CC\mid \varphi(x)=x=(\pi\otimes\iota)(x)\}$ has dimension $4$ over $\RR$. Since $V_h=V_e\xi$, $V_{h^2}=V_e\xi^2$, and $\dim V_e=8$, we have $V=V_e\oplus V_h\oplus V_{h^2}$, so the support of $\Gamma$ is $\langle h\rangle$. It follows that $\dim \cC_h=2=\dim \cC_{h^2}$, and $\cC_h$ and $\cC_{h^2}$ are isotropic subspaces of $\cC$ paired by the norm. Hence $\cC$ is split and $\Gamma_\cC$ is isomorphic to $\Gamma_{\OO_s}(G,\gamma)$ with $\gamma=(e,h,h^2)$.

\bigskip

The next result, which uses Lemma \ref{le:h_h2} and Theorem \ref{th:typeIII_C}, summarizes our arguments. We use the notation of Corollary \ref{co:real_Cayley} for the nontrivial gradings on Cayley algebras and write $\Gamma_\cC^\mathrm{triv}$ for the trivial grading on $\cC$.

\begin{theorem}\label{th:typeIII_twisted}
Let $G$ be an abelian group and let $(V,\LL,\beta,Q)$ be a real twisted composition of rank $8$ with $\LL=\RR\times\CC$. Then any $G$-grading of Type~III on $(V,\LL,\beta,Q)$ is isomorphic to one of the following gradings, 
for some element $h\in G$ of order $3$:
\begin{enumerate}
\item[(1.a)] 
$\Gamma\bigl(G,\Gamma_\OO(G,\RR,T),h\bigr)$ for an elementary abelian subgroup $T$ of order $8$ in $G$;
\item[(1.b)] 
$\Gamma\bigl(G,\Gamma_{\OO_s}(G,\RR,T,\mu),h\bigr)$ for an elementary abelian subgroup $T$ of order $8$ in $G$ and a nontrivial group homomorphism $\mu\colon T\rightarrow\{\pm 1\}$;
\item[(2.a)] 
$\Gamma\bigl(G,\Gamma_\OO(G,\CC,T),h\bigr)$ for an elementary abelian subgroup $T$ of order $4$ in $G$;
\item[(2.b)] 
$\Gamma\bigl(G,\Gamma_{\OO_s}(G,\CC,T,\mu),h\bigr)$ for an elementary abelian subgroup $T$ of order $4$ in $G$ and a nontrivial group homomorphism $\mu\colon T\rightarrow\{\pm 1\}$;
\item[(2.c)] 
$\Gamma\bigl(G,\Gamma_{\OO_s}(G,\gamma),h\bigr)$ with $\gamma=(g_1,g_2,g_3)\in (G\setminus\langle h\rangle)^3$ and $g_1g_2g_3=e$;
\item[(4.a)] 
$\Gamma\bigl(G,\Gamma_\OO(G,\HH,T),h\bigr)$ for a subgroup $T$ of order $2$ in $G$;
\item[(4.b)] 
$\Gamma\bigl(G,\Gamma_{\OO_s}(G,\HH,T),h\bigr)$ for a subgroup $T$ of order $2$ in $G$;
\item[(4.c)]
$\Gamma\bigl(G,\Gamma_{\OO_s}(G,\gamma),h\bigr)$ with $\gamma=(e,g,g^{-1})$, $g\in G\setminus\langle h\rangle$;
\item[(8.a)]
$\Gamma\bigl(G,\Gamma_\OO^\mathrm{triv},h\bigr)$; 
\item[(8.b)]
$\Gamma\bigl(G,\Gamma_{\OO_s}^\mathrm{triv},h\bigr)$;
\item[(8.c)]
$\Gamma\bigl(G,\Gamma_{\OO_s}(G,\gamma),h\bigr)$ with $\gamma=(e,h,h^2)$.
\end{enumerate}
Two $G$-gradings $\Gamma$ and $\Gamma'$, as above, are not isomorphic if they belong to different items. If $\Gamma$ and $\Gamma'$ belong to one item other than \textup{(2.c)}, then they are isomorphic if and only if $\langle h\rangle=\langle h'\rangle$ and, whenever applicable, $T=T'$, $\mu=\mu'$, and $g'\in\{g^{\pm 1}\}$.
If $\Gamma$ and $\Gamma'$ belong to item \textup{(2.c)}, then they are isomorphic if and only if $\langle h\rangle=\langle h'\rangle$ and there exist a permutation $\pi\in \textup{Sym}(3)$, $j\in\{0,1,2\}$ and $k\in\{\pm 1\}$ such that $g_i'=g^k_{\pi(i)}h^j$ for $i=1,2,3$.
\qed
\end{theorem}

Finally, we consider fine gradings. All $G$-gradings in items (1.a), (2.a), (4.a), and (8.a) are induced from the $\wt{G}$-grading $\Gamma\bigl(\wt{G},\Gamma_\OO\bigl(\wt{G},\RR,\ZZ_2^3),\tilde{h}\bigr)$, where $\wt{G}=\ZZ_2^3\times\ZZ_3$ and $\tilde{h}=(\bar 0,\bar 0,\bar 0,\bar 1)$, by means of a suitable homomorphism $G\to\wt{G}$. All gradings in items (1.b), (2.b), (4.b), and (8.b) are induced from $\Gamma\bigl(\wt{G},\Gamma_{\OO_s}\bigl(\ZZ_2^3\times\ZZ_3,\RR,\ZZ_2^3,\tilde\mu),\tilde{h}\bigr)$, where $\wt{G}$ and $\tilde{h}$ are as above, and $\tilde\mu\colon\ZZ_2^3\rightarrow\{\pm 1\}$ is the homomorphism $(\bar 1,\bar 0,\bar 0)\mapsto 1$, $(\bar 0,\bar 1,\bar 0)\mapsto 1$, $(\bar 0,\bar 0,\bar 1)\mapsto -1$ (see Corollary \ref{co:real_fine_Cayley}). All gradings in items (2.c), (4.c) and (8.c) are induced from $\Gamma\bigl(\wt{G},\Gamma_{\OO_s}(\wt{G},\tilde\gamma),\tilde{h}\bigr)$, where this time $\wt{G}=\ZZ^2\times\ZZ_3$, $\tilde{h}=(0,0,\bar 1)$, and $\tilde\gamma=\bigl((1,0,\bar 0),(0,1,\bar 0),(-1,-1,\bar 0)\bigr)$ (which gives the Cartan grading on $\OO_s$). The next result is now clear (compare with Corollary \ref{co:real_fine_Cayley}).

\begin{theorem}\label{th:typeIII_fine}
Consider the real twisted composition $\TC(\overline{\cC},\LL)$ where $\cC$ is a real Cayley algebra and $\LL=\RR\times\CC$. 
\begin{enumerate}
\item If $\cC=\OO$, then the twisted composition admits only one fine grading of Type~III, up to equivalence. The universal group of this grading is $\ZZ_2^3\times\ZZ_3$.
\item If $\cC=\OO_s$, then the twisted composition admits two fine gradings of Type~III, up to equivalence. The universal groups of these gradings are $\ZZ_2^3\times\ZZ_3$ and $\ZZ^2\times\ZZ_3$.\qed
\end{enumerate}
\end{theorem}

%
%

\section{Type~III gradings on simple real Lie algebras of type $D_4$}\label{s:Type_III_fine}

In view of Theorems \ref{th:real_tri_to_comp} and \ref{th:tri_isomorphism_2D4} and the fact that the restriction map gives an isomorphism of the automorphism group schemes of a trialitarian algebra $E$ and of the corresponding Lie algebra $\cL(E)$ of type $D_4$ (see Subsection \ref{ss:trialitarian_D4}), 
Theorem~\ref{th:typeIII_twisted} gives us the classification, up to isomorphism, of Type~III gradings on the central simple real Lie algebras of type $D_4$ as follows. 
For a twisted composition $(V,\LL,\beta,Q)$ of rank $8$ with $\LL=\RR\times\CC$ and each $G$-grading $\Gamma$ as in Theorem \ref{th:typeIII_twisted}, we take the induced $G$-grading on the trialitarian algebra $E=\End_\LL(V)$ and restrict it to $\cL=\cL(E)$, thus obtaining representatives of all isomorphism classes of Type~III gradings on $\cL$. The purpose of this section is to show what these gradings on $\cL$ look like.

As before, let $\omega=-\frac12+\frac{\sqrt{3}}{2}\bi$ (a cube root of unity in $\CC$) and $\xi=(1,\omega)\in\LL$. 
Let $\iota$ be the generator  of $\Gal(\CC/\RR)$ (the complex conjugation). The map
\[
\psi\colon\LL\otimes\CC\rightarrow \CC\times\CC\times\CC,\quad (b,c)\otimes c'\mapsto (bc',cc',\bar{c}c'),
\]
is an isomorphism of $\CC$-algebras. The $\CC$-antilinear automorphism $\id\otimes\iota$ on $\LL\otimes\CC$ corresponds under $\psi$ to the automorphism
\[
\hat\iota\colon\CC\times\CC\times\CC\rightarrow \CC\times\CC\times\CC,\quad
(c_1,c_2,c_3)\mapsto (\bar{c}_1,\bar{c}_3,\bar{c}_2),
\]
that is, $\psi\circ(\id\otimes\iota)=\hat\iota\circ\psi$.

Let $\cC$ be a real Cayley algebra (so $\cC$ is isomorphic either to $\OO$ or to $\OO_s$)
and consider the twisted Hurwitz composition $(V,\LL,\beta,Q)=\TC(\overline{\cC},\LL)$, namely, 
\[
V=\{x\in \cC\otimes\LL\otimes \CC\mid \tilde\iota(x)=x\}=(\RR 1\otimes\LL\otimes 1)\oplus(\cC^0\otimes\LL\otimes\bi),
\]
as in Equation~\eqref{eq:VTC}, where $\tilde\iota=\bar{}\otimes\id\otimes\iota$, $Q$ is the extension of the norm $n$ of $\cC$ to $V$, and $\beta$ is obtained from the para-Cayley multiplication of $\overline{\cC}$ by extending to $\cC\otimes\LL\otimes\CC$ as shown in Equation~\eqref{eq:TC} and then restricting to $V$. It will be convenient 
to denote $\pi(x)=\bar x$ for $x\in \cC$.

The trialitarian algebra $\End_\LL(V)$ is the subalgebra of fixed points of the $\CC$-antilinear operator $\mathrm{Int}(\pi)\otimes(\id\otimes\iota)$ on $\End_{\LL\otimes\CC}\bigl(\cC\otimes(\LL\otimes\CC)\bigr)\simeq\End_\RR(\cC)\otimes (\LL\otimes\CC)$.
Under the natural isomorphisms of $\CC$-algebras,
\[
\begin{split}
\End_\RR(\cC)\otimes(\LL\otimes\CC)\xrightarrow{\id\otimes\psi}\End_\RR(\cC)\otimes\CC^3&\simeq\big(\End_\RR(\cC)\otimes\CC\big)^3\\
&\simeq \End_\RR(\cC)^3\otimes \CC,
\end{split}
\]
the operator $\mathrm{Int}(\pi)\otimes(\id\otimes\iota)$ corresponds to $\tau\otimes\iota$ on $\End_\RR(\cC)^3\otimes\CC$, where 
\[
\tau(f_1,f_2,f_3)=(\bar f_1,\bar f_3,\bar f_2),
\]
with $\bar f=\mathrm{Int}(\pi)(f)=\pi f\pi\colon x\mapsto \overline{f(\bar x)}$.

Recall the triality Lie algebra (Definition~\ref{df:tri} with $\cS=\wb{\cC}$):
\begin{equation*} 
\tri(\cC)=\{(f_1,f_2,f_3)\in\frso(\cC,n)^3\mid f_1(x\bullet y)=f_2(x)\bullet y+x\bullet f_3(y)\ \forall x,y\in \cC\}.
\end{equation*}
This is a Lie algebra under the componentwise bracket, it is closed under cyclic permutations, and each of the three projections $(f_1,f_2,f_3)\mapsto f_i$ is an isomorphism $\tri(\cC)\to\frso(\cC,n)$.

\begin{proposition}\label{pr:D4_tri}
Let $\cL=\cL(E)$ be the Lie algebra of type $D_4$ attached to the trialitarian algebra $E=\End_\LL(V)$ where $V=\TC(\overline{\cC},\LL)$.
Under the isomorphism $\End_\LL(V)\otimes\CC\simeq \End_\RR(\cC)^3\otimes\CC$, $\cL$ corresponds to the Lie algebra of the elements of 
$\tri(\cC)\otimes\CC$ fixed under $\tau\otimes\iota$. 
\end{proposition}

\begin{proof}
The Lie algebra of type $D_4$ attached to the trialitarian algebra $\End_{\RR^3}(\cC\otimes\RR^3)$ is 
$\Der_{\RR^3}(\cC\otimes\RR^3,*,Q)$, which corresponds, under the isomorphism 
$\End_{\RR^3}(\cC\otimes\RR^3)\simeq\End_\RR(\cC)^3$, to the triality Lie algebra $\tri(\cC)$ (see Subsection \ref{ss:cyclic_compositions}). Extending the ground field from $\RR$ to $\CC$, we can identify 
$\Der_{\CC^3}(\cC\otimes\CC^3,*,Q)$ with $\tri(\cC)\otimes\CC$.

On the other hand, the Lie algebra $\cL$ is $\Der_\LL(V,\beta,Q)$, which is the subalgebra of fixed points in $\Der_{\LL\otimes\CC}(V\otimes\CC,\beta,Q)\subset\End_{\LL\otimes\CC}\big(\cC\otimes(\LL\otimes\CC)\big)\simeq\End_\RR(\cC)\otimes(\LL\otimes\CC)$ of the operator $\mathrm{Int}(\pi)\otimes(\id\otimes\iota)$. The result follows.
\end{proof}

As a module over $\LL=\RR\times\CC$, we can write $V=V_0\times V_1$ where
\[
V_0=\{x\in \cC\otimes(1,0)\otimes\CC\mid \tilde\iota(x)=x\}
=\RR(1\otimes(1,0)\otimes 1)\oplus \bigl(\cC^0\otimes(1,0)\otimes \bi\bigr),
\]
which is canonically isomorphic to 
\[
\tilde{V}_0\bydef\RR(1\otimes 1)\oplus \bigl(\cC^0\otimes\bi\bigr),
\]
the subspace of the elements of $\cC\otimes\CC$ fixed by $\pi\otimes\iota$.
Consider the projection onto the first component $\textup{pr}\colon\End_\LL(V)\rightarrow \End_\RR(V_0)$. 

\begin{corollary}\label{co:D4_so} 
The projection $\textup{pr}\colon\End_\LL(V)\to\End_\RR(V_0)\simeq\End_\RR(\tilde{V}_0)$ restricts to an isomorphism 
from $\cL$ onto $\frso(\tilde V_0,n)$, and the latter can be identified with the following $\RR$-form of 
$\frso(\cC\otimes \CC,n)$:
\[
\{f\in\frso(\cC\otimes \CC,n)\mid (\pi\otimes\iota)f=f(\pi\otimes\iota)\}.
\] 
\end{corollary}

\begin{proof}
After extending the ground field from $\RR$ to $\CC$, the projection $\textup{pr}$ corresponds, under the isomorphisms
$\End_\LL(V)\otimes\CC\simeq\End_\CC(\cC\otimes\CC)^3$ and $\End_\RR(V_0)\otimes\CC\simeq\End_\CC(\cC\otimes\CC)$, to the 
projection of $\End_\CC(\cC\otimes\CC)^3$ onto the first component. 

Under the isomorphism $\End_\RR(\cC)^3\otimes\CC\simeq\End_\CC(\cC\otimes\CC)^3$, the operator $\tau\otimes\iota$ 
corresponds to the mapping
$(f_1,f_2,f_3)\mapsto (\bar f_1,\bar f_3,\bar f_2)$ on $\End_\CC(\cC\otimes\CC)^3$, 
where $\bar f=\mathrm{Int}(\pi\otimes\iota)(f)=(\pi\otimes\iota)f(\pi\otimes\iota)$, so
$\cL$ corresponds to the fixed points of this mapping in $\tri(\cC\otimes\CC)$.
Hence, the projection onto the first component gives an isomorphism from $\cL$ onto the subalgebra of fixed points of 
$\pi\otimes\iota$ in $\frso(\cC\otimes\CC,n)$, which equals
$\{f\in\frso(\cC\otimes\CC,n)\mid f(\tilde V_0)\subseteq \tilde V_0\}$,
so it can be identified with $\frso(\tilde V_0,n)$.
\end{proof}

\begin{remark}\label{rem:inertia}
The quadratic space $(\cC,n)$ has inertia $(8,0)$ for $\cC\simeq\OO$ and $(4,4)$ for $\cC\simeq\OO_s$, hence 
the quadratic form on $\tilde V_0$ has inertia $(1,7)$ and $(5,3)$, respectively. As pointed out in the Introduction, these are the only cases in which the Lie algebra $\frso_{p,q}(\RR)$ can have Type~III gradings.
\end{remark}

We will now obtain a finer description of $\cL$ needed to deal with Type~III gradings. 
For $x\in \cC^0$, let $L_x$ and $R_x$ denote the left and right multiplications by $x$, and let $\ad_x=L_x-R_x$ and $T_x=L_x+R_x$. Note that $\pi L_x=-R_x\pi$, so $\pi\ad_x\pi=\ad_x$ and $\pi T_x\pi=-T_x$. Besides, any derivation $d\in\Der_\RR(\cC)$ commutes with $\pi$.

\begin{lemma}\label{le:D4_Der_ad_T}
Under the isomorphism $\End_\LL(V)\otimes\CC\simeq\End_\RR(\cC)^3\otimes\CC$, the Lie algebra $\cL$ corresponds to the direct sum
\begin{multline}\label{eq:D4_Der_ad_T}
\left(\{(d,d,d)\mid d\in\Der_\RR(\cC)\}\otimes 1\right)\\
\oplus\left(\{(\ad_x,-2L_x-R_x,L_x+2R_x)\mid x\in \cC^0\}\otimes 1\right)
\oplus\left(\{(T_x,-R_x,-L_x)\mid x\in \cC^0\}\otimes\bi\right).
\end{multline}
\end{lemma}

\begin{proof}
By Proposition \ref{pr:D4_tri}, we must compute the elements of $\tri(\cC)\otimes\CC$ fixed by $\tau\otimes\iota$. For $x,y,z\in \cC$, consider the \emph{associator} $(x,y,z)\bydef (xy)z-x(yz)$, which is an alternating function of $x$, $y$ and $z$. For any $x\in \cC^0$ and $u,v\in \cC$, we have:
\[
\begin{split}
L_x(u\bullet v)&=x(\bar u\bar v)=-(x,\bar u,\bar v)+(x\bar u)\bar v\\
  &=(\bar u,x,\bar v)+(x\bar u)\bar v\\
  &=(\bar ux+x\bar u)\bar v-\bar u(x\bar v)\\
  &=-T_x(u)\bullet v+u\bullet R_x(v),
\end{split}
\]
so $(L_x,-T_x,R_x)\in\tri(\cC)$, hence also $(R_x,L_x,-T_x)$ and $(T_x,-R_x,-L_x)$ are in $\tri(\cC)$, as well as
$(\ad_x,-T_x-L_x,R_x+T_x)=(\ad_x,-2L_x-R_x,L_x+2R_x)$. But, from \cite[(3.76)]{Schafer}, we have
\[
\frso(\cC,n)=\Der_\RR(\cC)\oplus L_{\cC^0}\oplus R_{\cC^0}=\Der_\RR(\cC)\oplus \ad_{\cC^0}\oplus T_{\cC^0},
\]
so we obtain
\begin{multline*}
\tri(\cC)=\{(d,d,d)\mid d\in\Der_\RR(\cC)\}\\
\oplus\{(\ad_x,-2L_x-R_x,L_x+2R_x)\mid x\in \cC^0\}
\oplus \{(T_x,-R_x,-L_x)\mid x\in \cC^0\}.
\end{multline*}
The result follows because we have $\pi d\pi=d$ for all $d\in \Der_\RR(\cC)$, and 
$\pi\ad_x\pi=\ad_x$, $\pi(-2L_x-R_x)\pi=L_x+2R_x$,
$\pi T_x\pi=-T_x$ and $\pi R_x\pi=-L_x$ for all $x\in \cC^0$.
\end{proof}

\begin{corollary}\label{co:D4_Der_ad_T}
Under the isomorphism $\End_\LL(V)\otimes\CC\simeq \End_\RR(\cC)\otimes (\LL\otimes\CC)$, 
the Lie algebra $\cL$ corresponds to
\begin{align*}
\bigl(\Der_\RR(\cC)\otimes 1\otimes 1\bigr)&\oplus\{\ad_x\otimes\xi\otimes 1+\sqrt{3}\,T_x\otimes\xi\otimes \bi \mid x\in \cC^0\}\\
&\oplus\{\ad_x\otimes\xi^2\otimes 1-\sqrt{3}\,T_x\otimes \xi^2\otimes \bi\mid x\in \cC^0\}.
\end{align*}
\end{corollary}

\begin{proof}
$\cL$ corresponds to the Lie algebra in Equation~\eqref{eq:D4_Der_ad_T} under the isomorphism $\End_\LL(V)\otimes\CC\simeq \End_\RR(\cC)^3\otimes\CC$. On the other hand, under the isomorphism 
$\End_\RR(\cC)\otimes(\LL\otimes\CC)\simeq \big(\End_\RR(\cC)\otimes\CC\big)^3$, an element $f\otimes\xi^i\otimes c$ corresponds to $(f\otimes c,f\otimes\omega^i c,f\otimes \omega^{2i}c)$ for $i=0,1,2$. 
Hence, $\ad_x\otimes\xi\otimes 1+\sqrt{3}\,T_x\otimes \xi\otimes\bi$ corresponds to
\begin{multline*}
\Bigl(\ad_x\otimes 1 +\sqrt{3}\,T_x\otimes\bi,\,\ad_x\otimes\omega +\sqrt{3}\,T_x\otimes\bi\omega,\,
\ad_x\otimes\omega^2+\sqrt{3}\,T_x\otimes\bi\omega^2\Bigr)\\
=\Bigl(\ad_x\otimes 1+\sqrt{3}\,T_x\otimes\bi,\,(-2L_x-R_x)\otimes 1-\sqrt{3}\,R_x\otimes\bi,\,
(L_x+2R_x)\otimes 1-\sqrt{3}\,L_x\otimes\bi\Bigr),
\end{multline*}
which, under the isomorphism $\big(\End_\RR(\cC)\otimes\CC\big)^3\simeq\End_\RR(\cC)^3\otimes\CC$, corresponds to
\begin{equation}\label{eq:triples}
(\ad_x,-2L_x-R_x,L_x+2R_x)\otimes 1+\sqrt{3}\,(T_x,-R_x,-L_x)\otimes\bi.
\end{equation}
A similar computation applies to $\ad_x\otimes\xi^2\otimes 1-\sqrt{3}\,T_x\otimes \xi^2\otimes\bi$ (with $\sqrt{3}$ being replaced by $-\sqrt{3}$). The result follows.
\end{proof}

\begin{theorem}\label{th:TypeIII_so}
Let $G$ be an abelian group, $\cC$ a real Cayley algebra endowed with a $G$-grading $\Gamma_\cC$, and $h$ an element of order $3$ in $G$. Consider the corresponding grading $\Gamma=\Gamma(G,\Gamma_\cC,h)$ of Type~III on the twisted Hurwitz composition $\TC(\overline{\cC},\LL)$. The grading induced by $\Gamma$ on the associated Lie algebra $\cL$ of type $D_4$ corresponds, under the isomorphism $\cL\simeq\frso(\tilde V_0,n)$ of Corollary \ref{co:D4_so}, to the $G$-grading on $\frso(\tilde V_0,n)$ whose homogeneous component of degree $g\in G$ is given by
\begin{align*}
\frso(\tilde V_0,n)_g=\Bigl(\Der_\RR(\cC)_g\otimes 1\Bigr)
&\oplus\left\{\ad_x\otimes 1+\sqrt{3}\,T_x\otimes\bi\mid x\in \cC^0_{gh^2}\right\}\\
&\oplus \left\{\ad_x\otimes 1-\sqrt{3}\,T_x\otimes\bi\mid x\in \cC^0_{gh}\right\}.
\end{align*} 
\end{theorem}

\begin{proof}
Recall that the grading $\Gamma$ on $\TC(\overline{\cC},\LL)$ is induced by $\Gamma_\cC$ on $\cC$ and the grading on $\LL$ defined by $\deg(\xi)\bydef h$. Hence, for any $x\in \cC^0_g$, the element $\ad_x\otimes\xi\otimes 1 +\sqrt{3}\,T_x\otimes\xi\otimes\bi$ is homogeneous of degree $gh$ in $\End_\RR(\cC)\otimes(\LL\otimes\CC)$, while the element 
$\ad_x\otimes\xi^2\otimes 1 -\sqrt{3}\,T_x\otimes\xi^2\otimes\bi$ is homogeneous of degree $gh^2$. 
Together with $d\otimes 1\otimes 1$ for $d\in\Der_\RR(\cC)_g$, these elements span 
the Lie algebra corresponding to $\cL$ under the isomorphism of Corollary~\ref{co:D4_Der_ad_T}.
More precisely, $\ad_x\otimes\xi\otimes 1 +\sqrt{3}\,T_x\otimes\xi\otimes\bi$ corresponds to the element of $\End_\RR(\cC)^3\otimes\CC$ given by \eqref{eq:triples}, and $\ad_x\otimes\xi^2\otimes 1 -\sqrt{3}\,T_x\otimes\xi^2\otimes\bi$ 
corresponds to a similar element.

Since $\tilde V_0=\RR(1\otimes 1)\oplus (\cC^0\otimes\bi)$, we may identify $\frso(\tilde V_0,n)$ with the subalgebra $\bigl(\Der_\RR(\cC)\otimes 1\bigr)\oplus\bigl(\ad_{\cC^0}\otimes 1\bigr)\oplus\bigl(T_{\cC^0}\otimes\bi\bigr)$ in $\frso(\cC,n)\otimes\CC\simeq\frso(\cC\otimes\CC,n)$, and Corollary~\ref{co:D4_so} completes the proof.
\end{proof}

With this result, it is now easy to describe explicitly, say, the fine gradings of Type~III on the central simple real Lie algebras of type $D_4$ that admit them. In view of Theorem~\ref{th:typeIII_fine} and Remark~\ref{rem:inertia}, $\frso_{7,1}(\RR)$, being isomorphic to the Lie algebra of derivations of $\TC(\overline{\OO},\RR\times\CC)$, admits a unique (up to equivalence) fine grading of Type~III, whose universal group is $\ZZ_2^3\times\ZZ_3$ and which has $14$ homogeneous components of dimension $1$ (coming from $\cC^0$) and $7$ of dimension $2$ (coming from $\Der_\RR(\cC)$). 
Similarly, $\frso_{5,3}(\RR)$, which is associated to $\TC(\overline{\OO_s},\RR\times\CC)$, admits two (up to equivalence) fine gradings of Type~III, with universal groups $\ZZ_2^3\times\ZZ_3$ and $\ZZ^2\times\ZZ_3$: the first analogous to the above and the second having $26$ homogeneous components of dimension $1$ and $1$ of dimension $2$.


\end{document}